\documentclass[12pt, twoside, letterpaper]{amsart}
\usepackage{amsmath, hyperref, color, amssymb}
\usepackage{srcltx, soul, cancel}

\usepackage{geometry}
\theoremstyle{plain}\newtheorem{Theorem}{Theorem}[section]\newtheorem{Corollary}[Theorem]{Corollary}\newtheorem{Lemma}[Theorem]{Lemma}\newtheorem{Proposition}[Theorem]{Proposition} %

\theoremstyle{definition}\newtheorem{Assumption}[Theorem]{Assumption}%
 \theoremstyle{remark}\newtheorem{Remark}[Theorem]{Remark} %

\newcommand{\Pas}{\text{$\mathbb{P}$--a.s.}}\newcommand{\esssup}{\operatorname*{\mathrm{ess\,sup}}}

\DeclareMathOperator{\cl}{cl}\DeclareMathOperator{\conv}{conv}

\begin{document}
\title[Optimal investment and consumption with labor income
]{
Optimal investment and consumption with labor income in incomplete markets
}

\author{Oleksii Mostovyi}\thanks{The first author has been supported by the National Science Foundation under grant No. DMS-1600307 (2015 - 2018), the second author supported by the National Science Foundation under grant No. DMS-1517664 (2015 - 2018). Any opinions, findings, and conclusions or recommendations expressed in this material are those of the authors and do not necessarily reflect the views of the National Science Foundation.}

\address{Oleskii Mostovyi, Department of Mathematics, University of Connecticut, Storrs, CT 06269, United States}%
\email{oleksii.mostovyi@uconn.edu}%

\author{Mihai S\^ irbu}
\address{Mihai S\^ irbu, Department of Mathematics, University of Texas at Austin, Austin, TX 78712, United States}
\email{sirbu@math.utexas.edu}%


\date{\today}%

\begin{abstract}
We consider the problem of optimal consumption from labor income and
  investment in a general incomplete semimartingale market. The economic
agent cannot borrow against future income, so the total wealth is required to be positive
at (all or some) previous times. Under very general conditions, we show that an optimal
consumption and investment plan exists and is unique, and provide a dual characterization
in terms of 
an optional strong supermartingale deflator 
and a decreasing part, which charges only the times when the no-borrowing constraint is binding. The analysis relies on the infinite-dimensional parametrization of the income/liability streams and, therefore, provides the first-order dependence of the optimal
investment and consumption plans on future income/liabilities (as well as a pricing rule).
An emphasis is placed on mathematical  generality.
%
\end{abstract}
\maketitle

\section{Introduction}
Optimal investment with intermediate consumption and a stream of labor income (or liabilities) is one of the central problems in mathematical economics. 
If borrowing against the future income is prohibited, the main  technical difficulty    lies in the fact that there are infinitely many constraints. Even in the deterministic case of no stocks and non-random income, a classical approach is based on the {\it convexification of the constraints} that leads to a non-trivial dual problem formulated over decreasing nonnegative functions.

Borrowing constraints imposed at all times  not only affect the  notion of admissibility, leading to more difficult mathematical analysis,  but also
change the meaning to fundamental concepts of mathematical finance such as replicability and completeness. The latter is   formulated via the attainability of every (bounded) contingent claim by a portfolio of traded  assets.   For a  labor income/liability streams that pays off dynamically, 
there is no a priori guarantee that such a replicating portfolio  (if it exists at all)  is admissible, i.e., satisfies the  constraints. Thus, in the terminology of \cite{HePages}, even a  complete market becomes dynamically incomplete under the borrowing constraints. 
 The  analysis of such a problem (in otherwise complete Brownian settings with a corresponding unique risk-neutral measure), is performed in \cite{HePages} and later in \cite{ElKarouiJeanblanc}. 
The  nonnegative decreasing {\it processes}   (that parametrize the dynamic incompleteness mentioned above)  play an important role in the characterizations of   optimal investment and consumption plans. The analysis  in \cite{HePages} and \cite{ElKarouiJeanblanc} is connected with optimal stopping techniques from~\cite{Kar89}. 

Incomplete markets with no-borrowing constraints have been analyzed only in specific Markovian models  in \cite{DuffieZarip1993} and \cite{DuffieFlemingSonerZarip1997} based on partial differential equations techniques. 
 The goal of the present paper is to study the problem of consumption and investment with no-borrowing constraints in general  (so, non-Markovian) incomplete   models. This leads to having, simultaneously, two layers of incompleteness. One comes from the many martingale measures, the other from a similar class of non-decreasing processes  (as above) that describe the dynamic incompleteness.  
 We refer to \cite{MalTrub07} for the examples of market incompleteness in finance and macroeconomics. 

In contrast to \cite{HePages} and \cite{ElKarouiJeanblanc},   our model not only allows for incompleteness, but also for jumps. Mathematically, this means we choose to work in a general semimartingale framework.
As in  \cite{HePages} and \cite{ElKarouiJeanblanc}, 
our approach is based on duality.  One of the principal difficulties is in the construction of the dual feasible set and the dual value function. It is well-known that the martingale measures   drive the dual domain in many problems of mathematical finance. On the other hand, the convexification of constraints leads to the decreasing processes as the central dual object as well. We show that the dual elements in the incomplete case can be approximated by  products of the densities of martingale measures and such nonnegative decreasing processes. This is one of the primary results of this work, see section \ref{secStructureOfDualDomain}, that leads to the complementary slackness characterization of optimal wealth in section \ref{secSlackness}, where it is shown that the approximating sequence for the dual minimizer leads to a nonincreasing process, which decreases at most when the constraints are attained. In turn, the dual minimizer can be written as a product of such a nondecreasing process and an optional strong supermartingale deflator. In the case of  complete Brownian markets  a similar result is proved in \cite{HePages} and   \cite{ElKarouiJeanblanc}.

In order to implement the approach, we {\it increase the dimensionality} of the problem and treat as arguments of the indirect utility not only the initial wealth, but also the function that specifies the number of units of labor income (or  the  stream of liabilities)  at any later time. This parametrization has the spirit of \cite{HK04}, however unlike \cite{HK04}, we go into (infinite-dimensional) non-reflexive spaces, which gives both novelty and technical difficulties to our analysis.
 Also, our formulation permits to price by marginal rate of substitution the whole  labor income process. This is done 
through  the subdifferentiability results in section \ref{secSubdifferentiability}.  Note that  the sub differential elements (prices) are time-dependent, so infinite-dimensional, unlike in  \cite{HK04}.

Another contribution of the paper lies in the {\it unified framework of admissibility} outlined in section \ref{secModel}.  More precisely, we assume that no-borrowing constraints are imposed  starting from  some pre-specified stopping time and hold up to the terminal time horizon.  
This framework allows us to treat in one formulation both the problem of no-borrowing constraints at all times (described above) and the one where borrowing against the future income is permitted  with a constraint only at the end. The latter is well-studied in the literature, see \cite{CvitanicSchachermayerWang}, \cite{Karatzas-Zitkovic-2003}, \cite{HK04}, \cite{Zitkovic}, and \cite{MostovyiRandEnd}.
 In such a formulation, the constraints reduce to a single inequality and the decreasing processes in the dual feasible set become constants.  
 
Among the many possibilities of constraints, \cite{Cuoco1997}  considers the problem of investment and consumption with labor income and no-borrowing constraints in Brownian market, even allowing for incompleteness. The dual problem cannot be solved directly (in part because for these constraints the dual space considered is too small) but the primal can be solved with direct methods. An approximate dual sequence can then be recovered from the primal. We generalize \cite{HePages} and \cite{ElKarouiJeanblanc} (complete Brownian markets) and \cite{Cuoco1997} (possibly incomplete Brownian markets) to the case of general semi-martingale incomplete markets. Our dual approach allows us, at the same time to obtain a dual characterization (complementary slackness) of the optimal consumption plan (not present in \cite{Cuoco1997}), similar to the complete case in \cite{HePages} and \cite{ElKarouiJeanblanc} and the possibility to study the dependence on labor income streams, through the parametrization of such streams.
 

Embedding path dependent problems into the convex duality framework have been analyzed in \cite{XiangYuHabit}, \cite{HaoMat}, \cite[Section 3.3]{Pham01}, whereas without duality but with random endowment it is considered in \cite{Miklos2018}, in the abstract singular control setting the duality approach is investigated in \cite{BankKaupila}.
Our embedding does not require any condition on labor income {replicability}, which becomes highly technical in the presence of extra admissibility constraints. Even in the case where the only constraint is imposed at maturity, in this part our approach differs from the one in \cite{HK04}, where non-replicability of the endowment (in the appropriate sense) is used in the proofs as it ensures that the effective domain of the dual problem has the same dimensionality as the primal domain. Note that even if 
 the  labor income is spanned by the same sources of randomness as the stocks, the idea of replicating the labor income and then reducing the problem to the one without it, does not necessarily work under the borrowing constraints, see the discussion in \cite[pp. 671-673]{HePages}.

 Some of the more specific   technical contributions  of this paper can be summarized as follows:

\begin{itemize}
\item We  analyze the {\it boundary behavior} of the value functions.  Note that the value functions are defined over infinite-dimensional spaces.

\item The finiteness of the indirect utilities {\it without} labor income is imposed  only, as a necessary and sufficient condition that allows for the standard conclusions of the utility maximization theory, see \cite{MostovyiNec}.

\item We show existence-uniqueness results for the {\it unbounded} labor income both from above and below.

%

\item We observe that the ``Snell envelope proposition'' \cite[Proposition 4.3]{K96} can be extended to 
 the envelope over all stopping times that exceed a  given initial stopping time $\theta _0$.
\item We represent the dual value function in terms of {\it uniformly integrable} densities of martingale measures, i.e., the densities of martingale measures under which the maximal wealth process of a self-financing portfolio that superreplicates the labor income, is a uniformly integrable martingale, see Lemma \ref{finitenessOverZ'} below. 

\end{itemize}

%
%

{\bf Organization of the paper}. In Section \ref{secModel}, we specify the model. We state and prove existence, uniqueness, semicontinuity and biconjugacy results in Section \ref{secProofs}, subdifferentiability is proven in Section \ref{secSubdifferentiability}. Structure of the dual domain is analyzed in Section \ref{secStructureOfDualDomain} and complimentary slackness is established in Section \ref{secSlackness}.
\section{Model}\label{secModel}

We consider a {financial market model} with finite time horizon
$[0,T]$ and a zero interest rate. The price process $S=(S^i)_{i=1}^d$
of the stocks is assumed to be a 
semimartingale on a complete stochastic basis
$\left( \Omega, \mathcal F, \left(\mathcal F_t \right)_{t\in[0,T]}, \mathbb
P\right)$, where $\mathcal F_0$ is trivial. 

Let $(e)_{t\in[0,T]}$ be an optional process that specifies the labor income rate, which is assumed to follow a certain stochastic clock, that we specify below.
Both processes $S$ and $e$ are given exogenously. 

We define a \textit{stochastic clock} as a
nondecreasing, c\`adl\`ag, adapted process such that
\begin{equation}
 \tag{finClock}
  \label{finClock}
\kappa_0 = 0, ~~ \mathbb P\left[\kappa_T>0 \right]>0,\text{ and } \kappa_T\leq A
\end{equation}
for some finite constant $A$. We note that the stochastic clock allows to include multiple standard formulations of the utility maximization problem in one formulation, see e.g., \cite[Example 2.5 - 2.9]{MostovyiNec}.
Let us define 
\begin{equation}\label{defK}
K_t {:=} \mathbb E\left[\kappa_t \right],\quad t\in[0,T].
\end{equation}
\begin{Remark}
The function $K$ defined in \eqref{defK} is right-continuous with left limits and takes values in $[0,A]$. 
\end{Remark}
We assume the income and consumption are given in terms of the clock $\kappa$.
Define a {portfolio} $\Pi$ as a quadruple $(x, q, H, c),$ where the
constant $x$ is the initial value of the portfolio, the function $q:[0,T]\to\mathbb R$ is a bounded and Borel measurable function,
 which 
specifies the amount of labor income rate, $H =
(H_i)_{i=1}^d$ is a predictable $S$-integrable process that
corresponds to the amount of each stock in the portfolio, and
$c=(c_t)_{t\in[0,T]}$ is the consumption rate, which we assume to be
 optional and nonnegative.

The \textit{wealth process} $V=(V_t)_{t\in[0,T]}$ {generated by the}
portfolio is 
\begin{equation}\nonumber
V_t = x + \int_0^t H_sdS_s + \int_0^t\left(q_se_s-c_s\right)d\kappa_s,\quad t\in[0,T].
\end{equation}
A portfolio $\Pi$ with $c\equiv 0$ and $q \equiv 0$ is called
\textit{self-financing}. The collection of nonnegative wealth
processes {generated by} self-financing portfolios with initial value $x\geq 0$ is
denoted by $\mathcal X(x)$, i.e.
\begin{equation}\nonumber
\mathcal X(x) {:=} \left\{X\geq 0:~X_t=x + \int_0^t
H_sdS_s,~~t\in[0,T] \right\},~~x\geq 0.
\end{equation}
A probability measure $\mathbb Q$ is an \textit{equivalent local
martingale measure} if $\mathbb Q$ is equivalent to $\mathbb P$ and
every $X\in\mathcal X(1)$ is a local martingale under $\mathbb Q$. We
denote the family of equivalent local martingale measures by
$\mathcal M$ and assume that
\begin{equation}\tag{noArb}\label{NFLVR}
\mathcal M \neq \emptyset.
\end{equation}
This condition is  equivalent to the absence of arbitrage
opportunities 
{in} the market
, see
 \cite{DS, DS1998} as well as 
\cite{KarKar07} for the exact statements and further
references.

To rule out doubling strategies in the presence of random endowment,
we need to impose additional restrictions. Following  \cite{DS97}, we say that a
nonnegative process in $\mathcal X(x)$ is \textit{maximal} if its
terminal value cannot be dominated by that of any other process in
$\mathcal X(x)$.
As in \cite{DS97}, we define an
\textit{acceptable} process to be a process of the form $X = X' -
X'',$ where $X'$ is a nonnegative wealth process {generated by} a self-financing portfolio and $X''$ is maximal.

 Our unified framework of admissibility is given by a  fixed stopping time $\theta_0$. The no-borrowing constraints will hold starting at this stopping time until the end.   Let $\Theta$ be the set of stopping times that are greater or equal than  $\theta_0$. 
\begin{Lemma}\label{3311}
Let $q^1$ and $q^2$ be bounded, Borel measurable functions on $[0,T]$, such that $q^1= q^2$, $dK$-a.e. Then, the cumulative labor income processes $\int_0^{\cdot}q^i_se_sd\kappa_s$, $i = 1,2$,  are indistinguishable. 
\end{Lemma}  The proof of this lemma is given in section \ref{subSecCharDualDomain}.
Following
 \cite{HK04}, we denote by
$\mathcal X(x,q)$ the set of acceptable processes with initial {values}
$x$, that dominate the labor income on $\Theta$:
\begin{equation}\nonumber
\begin{array}{rcl}
\mathcal X(x, q) &{:=}& \left\{ 
{\rm acceptable}~X: X_0 = x~{\rm and}\right.\\
&&\left.~ X_{\tau} + \int_0^{\tau}q_se_sd\kappa_s \geq 0, ~\mathbb P-a.s.~{\rm for~every~}\tau\in\Theta
\right\}.\\
\end{array}
\end{equation}
Let us set
\begin{equation}\nonumber
\mathcal K {:=} \left\{(x,q): \mathcal X(x,q) \neq \emptyset \right\},
\end{equation}
Let $\mathring{\mathcal K}$ denote the interior or $\mathcal K$ in the $\mathbb R\times \mathbb L^\infty(dK)$-norm topology.
We characterize $\mathcal K$ in Lemma~\ref{lemma1} below that in particular asserts that under Assumption \ref{asEnd}, $\mathring{\mathcal K}\neq \emptyset$. 
  
The set of admissible consumptions is defined as
\begin{equation}\nonumber
\begin{array}{rcl}
\mathcal A(x,q)&{:=} &\left\{ 
optional ~c\geq 0:~~there~exists~X\in\mathcal X(x,q), ~such~that~ \right.\\
&& 
\left.\int_0^{\tau}c_sd\kappa_s\leq X_{\tau} + \int_0^{\tau}q_se_sd\kappa_s,
~~for~every~\tau\in\Theta\right\},\quad (x,q)\in\mathcal K.\\ 
\end{array}
\end{equation}
Note that $c\equiv 0$ belongs to $\mathcal A(x,q)$ for every $(x,q)\in\mathcal K$. 
\begin{Remark}\label{Remark-Mihai} The no-borrowing constraints can also  be written  as
$$
\mathbb{P}\left (\int_0^{t}c_sd\kappa_s\leq X_{t} + \int_0^{t}q_se_sd\kappa_s,
~~for~every~ \theta _0\leq t\leq T \right)=1.$$
We write the constraints in terms of stopping times $\tau \in \Theta$ as 
we use the stopping times $\tau \in \Theta$ (and the corresponding decreasing processes that jump from one to zero at these  times) as  the building blocks of our analysis. 
\end{Remark}
\begin{Remark}\label{3312}
It follows from Lemma \ref{3311}, for every $x\in\mathbb R$,  we have
$$\mathcal X(x,q^1) = \mathcal X(x,q^2)\quad and \quad \mathcal A(x,q^1) = \mathcal A(x, q^2),$$
where some of these sets might be empty
\end{Remark}  

Hereafter, we shall impose the following conditions on the  endowment process.
\begin{Assumption}\label{asEnd}
There exists a maximal wealth process $X'$ such that
\begin{equation}\nonumber
X'_t \geq |e_t|, \quad for~every~t\in[0,T],\quad\Pas.
\end{equation}
Moreover, Assumption \ref{asEnd} and \eqref{NFLVR} imply that all the assertions of Lemma \ref{lemma1} hold.
\end{Assumption}

\begin{Remark}
If $\theta _0 = \{T\}$, then Assumption~\ref{asEnd} is equivalent to the assumptions on endowment in~\cite{HK04} (for the case of one-dimensional random endowment).
\end{Remark}


The preferences of an economic agent are modeled with a
\textit{utility stochastic field}
$U=U(t, \omega, x):[0,T]\times\Omega\times[0,\infty)\to\mathbb R\cup \{-\infty\}$.
We assume that $U$
satisfies the conditions below.
\begin{Assumption}\label{assumptionOnU}
For every $(t, \omega)\in[0, T]\times\Omega$, the function $x\to
  U(t, \omega, x)$ is strictly concave, increasing, continuously
  differentiable on $(0,\infty)$ and satisfies the Inada conditions:
  \begin{equation}\nonumber
    \lim\limits_{x\downarrow 0}U'(t, \omega, x) =
    \infty \quad \text{and} \quad \lim\limits_{x\to
      \infty}U'(t, \omega, x) = 0,
  \end{equation}
  where $U'$ denotes the partial derivative with respect to the third argument.
 At $x=0$ we suppose, by continuity, $U(t, \omega, 0) = \lim\limits_{x \downarrow 0}U(t, \omega, x)$, {which}
 may be $-\infty$.
For every
  $x\geq 0$ the stochastic process $U\left( \cdot, \cdot, x \right)$ is
  optional. Below, following the standard convention, we will not write $\omega$ in $U$.
\end{Assumption}

The agent can control investment and consumption. {The goal is to maximize}
 expected utility. The value function $u$ is
defined as:
\begin{equation}\label{primalProblem}
u(x, q) {:=} \sup\limits_{c\in\mathcal A(x, q)}\mathbb
E\left[\int_0^T U(t, c_t)d\kappa_t\right], \quad
(x, q)\in\mathcal K.
\end{equation}
In \eqref{primalProblem}, we use the convention
\begin{displaymath}
  \mathbb{E}\left[
    \int_0^TU(t,c_t)d\kappa_t \right] {:=} -\infty
  \quad \text{if} \quad \mathbb{E}\left[ \int_0^TU^{-}(t,c_t)d\kappa_t
  \right]= \infty.
\end{displaymath}
Here and below, $W^{-}$ and $W^{+}$ denote the negative and positive parts of a
stochastic field $W$, respectively.

 We employ duality techniques to obtain the standard conclusions of the utility maximization theory.  We first  define the convex conjugate stochastic field
\begin{equation}\label{defV}
V(t, y) {:=} \sup\limits_{x>0}\left(
U(t, x)-xy\right),\quad (t, y)\in[0,T]\times [0,\infty),
\end{equation}
 and then  observe that 
$-V$ satisfies Assumption \ref{assumptionOnU}.
In order to construct the feasible set of the dual problem, we define the set
$\mathcal L$ as the  polar cone of $-\mathcal K$:
\begin{equation}\label{defL}
\mathcal L {:=} \left\{(y,r)\in\mathbb R\times\mathbb L^1(dK): xy + \int_0^Tq_sr_sdK_s \geq
0{\rm~for~every~}(x,q)\in\mathcal K\right\}.
\end{equation}
\begin{Remark}
Under the conditions \eqref{finClock}, \eqref{NFLVR} and Assumption \eqref{asEnd}, the set 
$\mathcal L$ is non-empty.  By definition,  it is closed in $\mathbb R\times\mathbb L^1(d K)$-norm and $\sigma(\mathbb R\times\mathbb L^1(dK), \mathbb R\times\mathbb L^{\infty}(dK))$ topologies. Also,  as shown later, the set $\mathcal L = (-\mathring{\mathcal K})^o$, i.e. the polar of $-\mathring{\mathcal K}$.
\end{Remark}
By ${\mathcal Z}$, we denote the set of c\`adl\`ag densities of
equivalent local martingale measures:
\begin{equation}\label{defZ}
{\mathcal Z} {:=} \left\{{\rm c\grave adl\grave ag}~\left(\frac{d\mathbb Q_t}{d\mathbb P_t}\right)_{t\in[0,T]}:\quad \mathbb Q\in\mathcal M\right\}.
\end{equation}
Let us denote by $\mathbb L^0 = \mathbb L^0(d\kappa \times \mathbb P)$ the linear space of (equivalence classes of) real-valued optional processes on the stochastic basis $(\Omega, \mathcal F, (\mathcal F_t)_{t\in[0,T]}, \mathbb P)$ which we equip with the topology  of convergence in measure $(d\kappa \times \mathbb P)$. 
For each $y \geq 0$ we define
\begin{equation}\label{oldY}
\begin{array}{c}
 \hspace{-15mm}\mathcal Y(y) {:=} {\cl}\left\{ Y: ~Y{\rm~is~c\grave{a}dl\grave{a}g~adapted~and~}\right. \\
 \hspace{35mm}\left.0\leq Y\leq yZ ~\left(d\kappa\times\mathbb
P\right){\rm
~a.e.~for~some~}Z\in{\mathcal Z}\right\},\\
\end{array}
\end{equation}
where the closure is taken in
$\mathbb L^0$. 
Now we are ready to set the
domain of the dual problem:
\begin{equation}\label{defY}
\begin{array}{rcl}
\mathcal Y(y,r)&{:=} &\left\{
 Y: Y \in\mathcal Y(y)
~{\rm and}\right.\\
&&
\quad\mathbb E\left[\int_0^Tc_sY_sd\kappa_s \right]\leq xy + \int_0^Tq_sr_sdK_s,\\
&&\left.\quad{\rm for~every}~(x,q)\in\mathcal K {\rm ~and~}c\in\mathcal A(x,q)\right\}\\
\end{array}
\end{equation}
Note that the definition \eqref{defY}  requires that every element of $\mathcal Y(y,r)$ is in $\mathcal Y(y)$, $y\geq 0$. Also, for every $(y,r)\in\mathcal L$, $\mathcal Y(y,r)\neq \emptyset$, since $0\in\mathcal Y(y,r)$. 

 We can now  state the dual optimization problem:
\begin{equation}\label{dualProblem}
v(y, r) {:=} \inf\limits_{Y\in\mathcal Y(y, r)}\mathbb
E\left[\int_0^T V(t, Y_t)d\kappa_t\right],\quad(y, r)\in\mathcal L,
\end{equation}
where we use the convention:
\begin{displaymath}
  \mathbb{E}\left[ \int_0^TV(t,Y_t
    )d\kappa_t \right] {:=} \infty
  \quad \text{if} \quad \mathbb{E}\left[ \int_0^TV^{+}(t, Y_t )d\kappa_t
  \right] = \infty.
\end{displaymath}
Also, we set 
\begin{equation}\label{3313}
v(y,r) {:=} \infty~for~(y,r)\in\mathbb R\times \mathbb L^1(dK)\backslash
 \mathcal L\quad and \quad u(x,q){:=} -\infty~for~(x,q)\in\mathbb R\times \mathbb L^{\infty}(dK)\backslash\mathcal K.\\ 
\end{equation}
With this definition, it will be shown below in Theorem  \ref {mainTheorem} that   $u<\infty$ and $v>-\infty$ everywhere,  so  $u$ and $v$ are proper   functions  in the language of convex analysis.  
Let us recall that in the absence of random endowment, the dual value function is defined as
\begin{equation}\nonumber
\tilde w(y) {:=} \inf\limits_{Y\in\mathcal Y(y)}\mathbb E\left[\int_0^TV(t, Y_s)d\kappa_s \right],\quad y>0,
\end{equation}
whereas the primal value function is given by
\begin{equation}\label{defw}
w(x) {:=} u(x,0),\quad x>0.
\end{equation}
\section{Existence, uniqueness, and biconjugacy}\label{secProofs}
\begin{Theorem}\label{mainTheorem}
Let (\ref{finClock}) and (\ref{NFLVR}), Assumptions~\ref{asEnd}
and \ref{assumptionOnU} hold true and
\begin{equation}\tag{finValue}\label{finValue}
w(x) > -\infty\quad for~every~x>0\quad and \quad 
 \widetilde w(y) <\infty\quad for~every~y>0.
 \end{equation}
 Then we have:

$(i)$ $u$ is finite-valued on $\mathring{\mathcal K}$ and $u<\infty$ on $\mathbb R\times \mathbb L^\infty(dK)$.  The dual value function $v$ satisfies  $v>-\infty$ on $\mathbb R\times\mathbb L^1(dK)$,  and  the set $\{v<\infty\}$ is a nonempty convex subset of $\mathcal L$, whose closure in $\mathbb R\times\mathbb L^1(dK)$ equals to $\mathcal L$.  

$(ii)$ $u$ is concave, proper, and upper semicontinuous with respect to the  norm-topology of $\mathbb R\times\mathbb L^\infty(dK)$ and the  weak-star topology $\sigma (\mathbb R\times\mathbb L^\infty(dK), \mathbb R\times\mathbb L^1(dK))$. For every $(x,q)\in \{u>-\infty\}$, there exists a unique solution to (\ref{primalProblem}). Likewise, $v$ is convex, proper, and lower semicontinuous with respect to the norm-topology of $\mathbb R\times \mathbb L^1(dK)$ and the the weak topology  $\sigma(\mathbb R\times\mathbb L^1(dK),\mathbb R\times\mathbb L^{\infty}(dK))$. For every $(y,r)\in \{v<\infty\}$, there exists a unique solution to (\ref{dualProblem}). 

$(iii)$ The functions $u$ and $v$ satisfy the biconjugacy relations
\begin{equation}\nonumber
\begin{array}{rcll}
u(x,q) &=& \inf\limits_{(y,r)\in \mathcal L}\left(v(y,r) + xy + \int_0^Tr_sq_sdK_s\right),& (x,q)\in \mathcal K,\\
v(y,r) &=& \sup\limits_{(x,q)\in \mathcal K}\left(u(x,q) - xy - \int_0^Tr_sq_sdK_s\right),& (y,r)\in \mathcal L.\\
\end{array}
\end{equation}
\end{Theorem}



\begin{Lemma}\label{lemma1}
Let (\ref{NFLVR}) and Assumption~\ref{asEnd} hold. Then we have:

$(i)$ for every $x>0$, $(x,0)$ belongs to $\mathring{\mathcal K}$ (in particular, $\mathring{\mathcal K}\neq \emptyset$),

$(ii)$ for every $q\in\mathbb L^\infty(dK)$, there exists $x>0$ such that $(x,q)\in\mathring{\mathcal K}$,

$(iii)$ $\sup\limits_{\mathbb Q\in\mathcal M}\mathbb {E^Q}\left[\int_0^T|e_s|d\kappa_s \right] <\infty$,

$(iv)$ there exists a nonnegative maximal wealth process $X^{''}$, such that
\begin{displaymath}
X^{''}_T \geq \int_0^T|e_s|d\kappa_s,\quad \Pas,
\end{displaymath}

$(v)$ there exists a nonnegative maximal wealth process $X^{''}$, such that
\begin{equation}\nonumber
X^{''}_t \geq \int_0^t|e_s|d\kappa_
s,\quad t\in[0,T],~\mathbb P-a.s.
\end{equation}

\end{Lemma}
\begin{proof}
First, via \cite[Proposition I.4.49]{JS} and \eqref{finClock}, we get
\begin{equation}\nonumber
\begin{array}{c}
\int_0^T |e_s|d\kappa_s \leq \int_0^T X'_sd\kappa_s = -\int_0^T\kappa_{s-}dX'_s + \kappa_TX'_T \\ \leq -\int_0^T\kappa_{s-}dX'_s +AX'_T =AX'_0 + \int_0^T(A - \kappa_{s-})dX'_s.\\
\end{array}
\end{equation}
Therefore, the exists a self-financing wealth process $\bar X$, such that
$$\int_0^T |e_s|d\kappa_s \leq \bar X_T.$$
Consequently, \cite[Theorem 2.3]{DS97} asserts the existence of a nonnegative {\it maximal} process $X^{''}$, such that
$$\int_0^T |e_s|d\kappa_s \leq  X_T^{''},$$
i.e.,  $(iv)$ holds. Therefore $(iii)$ is valid as well, by \cite[Theorem 5.12]{DS1998}. 

Relation $(iv)$ and \eqref{NFLVR} imply $(v)$.
To prove $(i)$ and $(ii)$, without loss of generality, we will suppose that in $(v)$, $X^{''}_0> 0.$ Let $q\in\mathbb L^{\infty}(dK)$ and $\varepsilon >0$ be fixed.  Let us define
\begin{equation}\nonumber
x(q){:=} \left(||q||_{\mathbb L^{\infty}(dK)} + 2\varepsilon\right)X^{''}_0.
\end{equation}
We claim that $(x(q), q)\in\mathring{\mathcal K}.$
Let 
us consider arbitrary 
\begin{equation}\label{454}
|x'|\leq \varepsilon\quad {\rm and} \quad q'\in\mathbb L^\infty(dK)
:~||q'||_{\mathbb L^{\infty}(dK)}\leq {\varepsilon},
\end{equation}  
and set
\begin{equation}\label{455}
\tilde X_t {:=} \left(||q||_{\mathbb L^{\infty}(dK)} + 2\varepsilon+x'\right)X^{''}_t,\quad t\in[0,T].
\end{equation} Then, by item $(v)$, we have
\begin{displaymath}
\begin{array}{c}
\tilde X_t\geq \left(||q||_{\mathbb L^{\infty}(dK)}+{\varepsilon}\right)X^{''}_t\geq \left(||q||_{\mathbb L^{\infty}(dK)}+{\varepsilon}\right)\int_0^t|e_s|d\kappa_s\geq -\int_0^t(q_s + q'_s)|e_s|d\kappa_s.
\end{array}
\end{displaymath}
We deduce that $\tilde X\in\mathcal X(\left(||q||_{\mathbb L^{\infty}(dK)} + 2\varepsilon+x'\right)X^{''}_0, q+q')$. In particular,  $$\mathcal X(\left(||q||_{\mathbb L^{\infty}(dK)} + 2\varepsilon+x'\right)X^{''}_0, q+q')\neq \emptyset.$$ As $x'$ and $q'$ are arbitrary elements satisfying \eqref{454}, we deduce that $(x(q), q)\in\mathring{\mathcal K}$. This proves $(ii)$.

In order to show $(i)$, first we observe that $\mathring{\mathcal K}$ is a convex cone. Therefore, it suffices to prove that, for a given $\varepsilon>0$, we have
\begin{equation}\label{2121}
\left(2\varepsilon X^{''}_0, 0\right)\in\mathring{\mathcal K}.
\end{equation} 
Again, let us consider $x'$ and $q'$ satisfying \eqref{454} and $\tilde X$ satisfying \eqref{455} for $q\equiv 0$. 
Then for every $t\in[0,T]$, we have
$$\tilde X_t\geq {\varepsilon}X^{''}_t\geq {\varepsilon}\int_0^t|e_s|d\kappa_s\geq -\int_0^tq'_s|e_s|d\kappa_s.$$
Thus, $\mathcal X(\left(2\varepsilon+x'\right)X^{''}_0, q')\neq \emptyset$.  Consequently, as $x'$ and $q'$ are arbitrary elements satisfying \eqref{454}, \eqref{2121} holds and so is $(i)$. 
 This completes the proof of the lemma. 
\end{proof}
\begin{Remark}
A close look at the proofs shows that the conclusions of Theorem \ref{mainTheorem} also hold if instead of Assumption \ref{asEnd}, we impose any of the equivalent assertions $(iii) - (v)$ of Lemma \ref{lemma1}.
\end{Remark}


\subsection{Characterization of the primal and dual domains}\label{subSecCharDualDomain}
The polar, $A^{o}$, of a nonempty subset $A$ of $\mathbb R\times\mathbb L^\infty(dK)$, is
 the subset of 
 $\mathbb R\times\mathbb L^1(dK)$, defined by
$$A^o{:=} \left\{(y,r)\in \mathbb R\times \mathbb L^1(dK):~xy + \int_0^T q_sr_sdK_s \leq 1, \quad for~every~(x,q)\in A\right\}.$$
The polar of a subset of $\mathbb R\times\mathbb L^1(dK)$ is defined similarly.
\begin{Proposition}\label{prop1} Under Assumption \ref{asEnd} and conditions  \eqref{finClock} and \eqref{NFLVR}, we have:

$(i)$ Let $(x,q)\in\mathbb R\times\mathbb L^{\infty}(dK)$. Then $c\in \mathcal A(x,q)$ (thus, $\mathcal A(x,q)\neq \emptyset$ so $(x,q)\in \mathcal K$ )  if and only if
\begin{equation}\nonumber
\mathbb E\left[\int_0^T c_sY_sd\kappa_s \right]\leq xy + \int_0^Tq_sr_sdK_s,
~~ {\rm for~every}~(y,r)\in \mathcal L{\rm ~ and }~ Y\in\mathcal Y(y,r).
\end{equation}

$(ii)$ Likewise, for $(y,r)\in \mathbb{R}\times \mathbb{L}^1 (dK)$ we have  $(y,r)\in \mathcal L$ and $Y\in \mathcal Y(y,r)$ if and only if
\begin{equation}\nonumber
\mathbb E\left[\int_0^T c_sY_sd\kappa_s \right]\leq xy + \int_0^Tq_sr_sdK_s,
~~ {\rm for~every}~(x,q)\in\mathcal K{\rm ~ and }~c\in\mathcal A(x,q).
\end{equation}

We also have $\mathcal K = (-\mathcal L)^{o} = \cl{\mathring{\mathcal K}}$, where the closure is taken both in norm $\mathbb R\times \mathbb L^1(dK)$ and $\sigma(\mathbb R\times\mathbb L^\infty(dK), \mathbb R\times\mathbb L^1(dK))$ topologies.
\end{Proposition}
\begin{Remark}\label{remNonEmpty}
It follows from Proposition \ref{prop1} that 
for every $(x,q)\in\mathcal K, \mathcal A(x,q)\neq \emptyset$ as $0\in\mathcal A(x,q)$. Likewise, 
for every $(y,r)\in \mathcal L, \mathcal Y(y,r)\neq \emptyset$ as $0\in\mathcal Y(y,r)$. Moreover, for every $(x,q)\in  \mathring{\mathcal K}$, each of the sets $\mathcal A(x,q)$ and $\bigcup\limits_{(y,r)\in \mathcal L: ~xy + \int_0^Tr_sq_sdK_s \leq 1}\mathcal Y(y,r)$ contain a strictly positive element, see Lemma \ref{6-12-1} below. 
\end{Remark}
The proof of Proposition~\ref{prop1} will be given via several lemmas.
Let 
\begin{equation}\nonumber
\begin{array}{l}
\mathcal M'~be~the~set~of~equivalent~local~martingale~measures,~under~which\\X^{''}~(from~Lemma~\ref{lemma1},~item~(v))~is~a~uniformly~integrable~martingale. 
\end{array}
\end{equation}
Note that by \cite[Theorem 5.2]{DS97}, $\mathcal M'$ is a nonempty, convex subset of $\mathcal M$, which is also dense in $\mathcal M$ in the total variation norm. 

\begin{Remark}
Even though the results in \cite{DS97} are obtained under the condition that $S$ is a locally bounded process, they also hold without local boundedness assumption, see the discussion in \cite[Remark 3.4]{HKS05}.
\end{Remark}

Let $Z'$ denote the set of the corresponding c\`adl\`ag densities,~i.e., 
\begin{equation}\label{defZ'}
\begin{array}{rcl}
\mathcal Z' &{:=}& \left\{{\rm c\grave adl\grave ag}~Z: Z_t=\mathbb E\left[ \frac{d\mathbb Q}{d\mathbb P}|\mathcal F_t\right],~t\in[0,T],~\mathbb Q\in\mathcal M'\right\}.\\
\end{array}
\end{equation}
 We also set
\begin{equation}\label{defUpsilon}
\Upsilon {:=} \left\{1_{[0,\tau]}(t),t\in[0,T]:~~\tau\in\Theta\right\}.
\end{equation}

\begin{Lemma}\label{6-16-1}
Let   the conditions of Proposition \ref{prop1} hold, $\mathbb Q\in\mathcal M'$, ${Z} = {Z}^{\mathbb Q}$ be the corresponding element of $\mathcal Z'$,  and $\Lambda \in \Upsilon$.  Then  there exists $r\in \mathbb L^1(dK)$, uniquely defined~by
\begin{equation}\label{6-25-1}
\int_0^tr_sdK_s = \mathbb E\left[\int_0^t\Lambda_s{Z}_se_sd\kappa_s \right],\quad t\in[0,T],
\end{equation}
such that $(1,r)\in \mathcal L$  and  ${Z}\Lambda = ({Z}_t\Lambda_t)_{t\in[0,T]}\in\mathcal Y(1,r)$.
\end{Lemma}
\begin{proof}
Let $(x,q)\in\mathcal K$ and $c\in\mathcal A(x,q)$. 
Then there exists  
$X\in\mathcal X(x,q)$, such that
\begin{equation}\label{246}
\int_0^\tau c_sd\kappa_s\leq X_\tau + \int_0^\tau q_se_sd\kappa_s,\quad for~every~\tau\in\Theta.
\end{equation}
In particular, \eqref{246} holds for the particular $\tau$, such that $\Lambda_t = 1_{[0,\tau]}(t)$, $t\in[0,T]$.  
Let  $G_t{:=} \int_0^t(q_se_s)^{+}d\kappa_s$, $t\in[0,T]$, where $(\cdot)^{+}$ denotes the positive part. By \cite[Theorem III.29, p. 128]{Pr}, $\int_0^tG_{s-}d{Z}_s$, $t\in[0,T]$, is a local martingale, so let $(\sigma_n)_{n\in\mathbb N}$ be its localizing sequence. Then by the monotone convergence theorem, integration by parts formula, and the optional sampling theorem, we get
\begin{equation}\label{281}
\begin{array}{c}
\mathbb E\left[\int_0^T {Z}_s(q_se_s)^{+}\Lambda_sd\kappa_s \right] = \lim\limits_{n\to\infty}\mathbb E\left[\int_0^{\sigma_n\wedge \tau} {Z}_sdG \right] \\
= \lim\limits_{n\to\infty} \mathbb E\left[{Z}_{\sigma_n\wedge \tau}G_{\sigma_n\wedge \tau}\right] 
-\lim\limits_{n\to\infty} \mathbb E\left[ \int_0^{\sigma_n\wedge \tau}G_{s-}d{Z}_s\right] 
\\=
\lim\limits_{n\to\infty} \mathbb E\left[{Z}_{\sigma_n\wedge \tau}G_{\sigma_n\wedge \tau} \right] = \lim\limits_{n\to\infty} \mathbb E^{\mathbb Q}\left[G_{\sigma_n\wedge \tau} \right] 
 =  \mathbb E^{\mathbb Q}\left[G_\tau \right],\\
\end{array}
\end{equation}
where in the last equality we used the monotone convergence theorem again.
Here finiteness of $\mathbb E^{\mathbb Q}\left[G_\tau \right]=\mathbb E^{\mathbb Q}\left[\int_0^\tau(q_se_s)^{+}d\kappa_s \right]$ follows from Assumption \ref{asEnd} via Lemma \ref{lemma1}, part (iii). 
In a similar manner, we can show that
\begin{equation}\label{283}
\begin{array}{rcl}
\mathbb E\left[\int_0^T(q_se_s)^{-}{Z}_s\Lambda_sd\kappa_s\right] &= &\mathbb E^{\mathbb Q}\left[ \int_0^\tau (q_se_s)^{-}d\kappa_s\right]<\infty,\\
\mathbb E\left[\int_0^Tc_s{Z}_s\Lambda_sd\kappa_s\right] &=& \mathbb E^{\mathbb Q}\left[ \int_0^\tau c_sd\kappa_s\right]<\infty.\\
\end{array}
\end{equation}
\cite[Lemma 4]{HK04} applies here and asserts that $X$ in \eqref{246} is a supermartingale under $\mathbb Q$. 
Therefore, from \eqref{246}, using \eqref{281} and \eqref{283}, and taking expectation under $\mathbb Q$, we get
\begin{equation}\label{6-22-2}
 \mathbb E\left[ \int_0^Tc_s\Lambda_s{Z}_sd\kappa_s\right]\leq x + \mathbb E\left[\int_0^Tq_se_s{Z}_s\Lambda_sd\kappa_s\right].
\end{equation}
Let us define
\begin{equation}\nonumber
R_t {:=} \mathbb E\left[\int_0^t\Lambda_s{Z}_se_sd\kappa_s \right],\quad t\in[0,T].
\end{equation}
Using the monotone class theorem, we obtain
\begin{equation}\label{6-25-2}
\mathbb E\left[\int_0^T\tilde q_se_s{Z}_s\Lambda_sd\kappa_s\right]
=
\int_0^T\tilde q_sdR_s,\quad for~every\quad\tilde q\in\mathbb L^\infty(dK).
\end{equation}
We claim that
$dR$ is absolutely continuous with respect to $dK$. First, using the $\pi-\lambda$ theorem, one can show that for every Borel-measurable subset $A$ of $[0,T]$, we have
$$K(A) = \mathbb E\left[\int_0^T 1_A(t)d\kappa_t \right]\quad and \quad R(A) = \mathbb E\left[\int_0^T1_A(t)\Lambda_t{Z}_te_td\kappa_t\right].$$
Thus, if for some $A$, $K(A) = 0$, then $\int_0^T 1_A(t)d\kappa_t=0$ a.s. and $\kappa^A_t {:=} \int_0^t 1_A(s)d\kappa_s$, $t\in[0,T]$, satisfies $\kappa^A_T = 0$ a.s. and $\int_0^T \Lambda_t{Z}_te_t d\kappa^A_t = \int_0^T\Lambda_t{Z}_te_t1_A(t)d\kappa_t = 0$ a.s. 


As $dR$ is absolutely continuous with respect to $dK$,  
there exists a unique $r\in\mathbb L^1(dK)$, such that \eqref{6-25-1}
holds.
Since the left-hand side in (\ref{6-22-2}) is nonnegative and since $(x,q)$ is an arbitrary element of $\mathcal K$, we deduce from the definition of $\mathcal L$, \eqref{defL},  that $(1,r)\in \mathcal L$. Finally, it follows from \eqref{6-22-2} and \eqref{6-25-2}
that ${Z}\Lambda \in\mathcal Y(1,r)$. This completes the proof of the~lemma.
\end{proof}
\begin{Remark}\label{Remark-Mihai-2}
The  natural convexification of  the set $\Upsilon$  consists of non-negative  left-continuous  decreasing and adapted processes $D$ such that $D_{\theta _0}=1$. In the context of utility-maximization constraints (and for $\theta _0=0$), this is follows  from  \cite{HePages} and \cite{ElKarouiJeanblanc}. 
We investigate convexification of the constraints  in  the later Sections \ref{secStructureOfDualDomain} and \ref{secSlackness}, where will extend Lemma \ref{6-16-1} to a more general set of decreasing processes than $\Lambda$ that drives the dual domain and that allows for the multiplicative decomposition of the dual minimizer.
\end{Remark}
\begin{Corollary}\label{cor1}
Let   the conditions of Proposition \ref{prop1} hold, $\mathbb Q\in\mathcal M'$, ${Z}$ be the c\'adl\'ag modification of the density process $\mathbb E\left[ \frac{d\mathbb Q}{d\mathbb P}|\mathcal F_t\right]$, $t\in[0,T]$. Then there exists $r\in\mathbb L^1(dK)$, such that
 ${Z}\in\mathcal Y(1,r)$,
where
$$
\int_0^tr_sdK_s = \mathbb E\left[\int_0^t{Z}_se_sd\kappa_s \right],\quad t\in[0,T],
$$
and $(1,r)\in \mathcal L$.
\end{Corollary}
\begin{proof}[Proof of Lemma \ref{3311}] Let us fix $\mathbb Q\in\mathcal M'$ and let $Z\in\mathcal Z'$ be the corresponding density process. As in Lemma \ref{6-16-1}, we can show that there exists $r\in\mathbb L^1(dK)$, such that for every bounded and Borel measurable function $q$ on $[0,T]$ we have
\begin{equation}\label{5192}
\int_0^tq_s r_sdK_s = \mathbb E\left[\int_0^t q_sZ_s|e_s|d\kappa_s \right],\quad t\in[0,T].
\end{equation}
Let $\bar q {:=} |q^1 - q^2|$, then as $\bar q = 0$, $dK$-a.e., we get
\begin{displaymath}
\mathbb E\left[\int_0^T \bar q_s|e_s|Z_sd\kappa_s \right] = \int_0^T \bar q_sr_sdK = 0.
\end{displaymath}
Therefore, using integration by parts and via \eqref{5192}, we obtain
$$
0 = \int_0^T \bar q_sr_sdK = \mathbb E\left[\int_0^T q_sZ_s|e_s|d\kappa_s \right] = \mathbb E^{\mathbb Q}\left[\int_0^T q_s|e_s|d\kappa_s \right].
$$
Consequently, $\int_0^T q_s|e_s|d\kappa_s = 0$, $\mathbb Q$-a.s., and by the equivalence of $\mathbb Q$ and $\mathbb P$, also $\mathbb P$-a.s. 
As, by construction $\int_0^t q_s|e_s|d\kappa_s = 0$, $t\in[0,T]$, is a nonnegative and non-decreasing process, whose terminal value is $0$,  $\mathbb P$-a.s., we conclude that it is indistinguishable from the $0$-valued process. The assertions of the lemma follows.

\end{proof}
\begin{Lemma}\label{9-5-1}
Let   the conditions of Proposition \ref{prop1} hold,  $(x,q)\in\mathcal K$, and $c$ is a nonnegative optional process. Then $c\in\mathcal A(x,q)$ if and only if
\begin{equation}\label{241}
\mathbb E\left[\int_0^Tc_s\Lambda_s{Z}_sd\kappa_s \right]\leq
x + \mathbb E\left[\int_0^Tq_se_s\Lambda_s{Z}_sd\kappa_s \right] = x + \int_0^Tq_sr_sdK_s,
\end{equation}
for every ${Z}\in\mathcal Z'$ and $\Lambda \in \Upsilon$, where $r$ is given by (\ref{6-25-1}).
\end{Lemma}
\begin{proof}
Let $c\in\mathcal A(x,q)$. Then for every ${Z}\in\mathcal Z'$ and $\Lambda \in \Upsilon$,  the validity of \eqref{241} follows from the definition of $\mathcal A(x,q)$, integration by parts formula and supermartingale property of every $X\in\mathcal X(x,q)$ under every $\mathbb Q\in\mathcal M'$, which in turn follows from \cite[Lemma 4]{HK04}. 

Conversely, let \eqref{241} holds for every ${Z}\in\mathcal Z'$ and $\Lambda \in \Upsilon$. 
Then, we have
\begin{equation}\label{243}
\mathbb E\left[\int_0^T(c_s-q_se_s){Z}_s\Lambda_sd\kappa_s \right]\leq x,
\end{equation}
which, in view of  the definition of $\Upsilon$ in \eqref{defUpsilon}, localization, and integration by parts, implies that
\begin{displaymath}
\sup\limits_{\mathbb Q\in\mathcal M', \tau\in \Theta}\mathbb {E^Q}\left[\int_0^{\tau}(c_s-q_se_s)d\kappa_s \right]\leq x,
\end{displaymath}
For $X^{''}$ given by Lemma \ref{lemma1}, item $(v)$, let us denote
\begin{equation}\label{244}
f_t{:=} ||q||_{\mathbb L^\infty(dK)} X^{''}_t +\int_0^{t}(c_s-q_se_s)d\kappa_s,\quad t\in[0,T].
\end{equation}
It follows from Assumption \ref{asEnd} and item $(v)$ of Lemma \ref{lemma1} 
 that 
 $f$ is a nonnegative process.
We observe that the proof of \cite[Proposition 4.3]{K96} goes through, if we only take stopping times in $\Theta$ and measures in $\mathcal M'$. This proposition allows to conclude that there exists a nonnegative c\`adl\`ag process $V$, such that 
\begin{equation}\label{defV2}
V_t = \esssup\limits_{\tau\in \Theta: \tau\geq t,\mathbb Q \in\mathcal M'}\mathbb {E^Q}\left[f_{\tau}|\mathcal F_t \right],\quad t\in[0,T],
\end{equation}
 which is a supermartingale for every $\mathbb Q \in\mathcal M'$. Therefore, by the density of $\mathcal M'$ in $\mathcal M$ in the norm topology of $\mathbb L^1(\mathbb P)$ and Fatou's lemma, $V$ is a supermartingale under every $\mathbb Q\in\mathcal M$. Moreover, $V_0$ satisfies  
$$V_0\leq x + ||q||_{\mathbb L^\infty(dK)} X^{''}_0,$$ by \eqref{defV2}, \eqref{243}, and by following the argument in the proof of Lemma \ref{6-16-1}. 

We would like to apply  
the optional decomposition theorem of F\"olmer and Kramkov, \cite[Theorem 3.1]{FK}. For this, we need to show that $V$ is a local supermartingale under every $\mathbb Q$, such that every $X\in\mathcal X(1)$ is a $\mathbb Q$-local supermartingale. However, $\mathcal M$ is dense in the set of such measures in the norm topology of $\mathbb L^1(\mathbb P)$, by the results of Delbaen and Schachermayer, see \cite[Proposition 4.7]{DS1998}. Therefore, the supermartingale property of $V$ under every such $\mathbb Q$ follows from Fatou's lemma and supermartingale property of $V$ under every $\mathbb Q\in\mathcal M$ established above. Therefore, by \cite[Theorem 3.1]{FK}, we get
$$V_t = V_0 + H\cdot S_t - A_t,\quad t\in[0,T],$$
where $A$ is a nonnegative increasing process that starts at $0$. 
Subtracting the constant $||q||_{\mathbb L^\infty(dK)} X^{''}_0$ from both sides of \eqref{244}, we get
\begin{equation}\nonumber\begin{array}{c}
\int_0^{\tau}(c_s-q_se_s)d\kappa_s = f_\tau -||q||_{\mathbb L^\infty(dK)} X^{''}_0 \leq V_\tau -||q||_{\mathbb L^\infty(dK)} X^{''}_0\\= V_0 -||q||_{\mathbb L^\infty(dK)} X^{''}_0 + H\cdot S_\tau-A_\tau \leq x + H\cdot S_\tau, \quad \tau\in\Theta,
\end{array}\end{equation}
where $x + H\cdot S$ is acceptable, since $V$ is nonnegative. 
Consequently, 
$c\in\mathcal A(x,q)$. This completes the proof of the lemma.
\end{proof}


\begin{proof}[Proof of Proposition~\ref{prop1}]
The assertions of item $(i)$ follow from Lemma~\ref{9-5-1}. 
It remains to show that the affirmations of item $(ii)$ hold.
Fix a $(y,r)\in \mathcal L$. If $Y\in\mathcal Y(y,r)$, $(ii)$ follows from the definition of $\mathcal Y(y,r)$. Conversely,
if $(ii)$ holds for a nonnegative process $Y$, then since $(x,0)\in \mathcal K$ for every $x>0$, we have 
$$\mathbb E\left[\int_0^Tc_sY_sd\kappa_s \right]\leq 1\quad {\rm for~every~}c\in\mathcal A\left(\frac{1}{y},0\right).$$
Via \cite[Proposition 4.4]{MostovyiNec}, we deduce that $Y\in\mathcal Y(y)$ and is such that $(ii)$ holds. Therefore,
$Y\in\mathcal Y(y,r)$.  

We have  $\mathring{\mathcal K}\neq \emptyset$, where the interior is taken with respect to the norm-topology. According to Proposition \ref{prop1} (i)  we have   also $\mathcal K = (-\mathcal L)^{o}$.  The set $\mathcal{K}$, as the polar of $\mathcal L$,  is convex and  closed both in  (strong) $\mathbb R\times \mathbb L^{\infty}(dK)$ and $\sigma(\mathbb R\times\mathbb L^\infty(dK), \mathbb R\times\mathbb L^1(dK))$ topologies. Having non-empty strong interior, we obtain $\mathcal K =  \cl{\mathring{\mathcal K}}$ where the closure is in the strong-topology. 
Since $\mathcal{K}$ is also closed in the weaker $\sigma(\mathbb R\times\mathbb L^\infty(dK), \mathbb R\times\mathbb L^1(dK))$ topology we obtain that 
$$\mathcal K = (-\mathcal L)^{o} = \cl{\mathring{\mathcal K}},$$ where the closure is taken in both
topologies.
\end{proof}

\subsection{Preliminary properties of the value functions, in particular finiteness. }

Let $\mathbf  L^0_{+}$ be the positive orthant of $\mathbf  L^0$. The polar of a set $A \subseteq \mathbf  L^0_{+}$ is defined as
\begin{displaymath} 
A^o{:=} \left\{ c\in\mathbf  L^0_{+}:~\mathbb E\left[\int_0^T c_sY_s d\kappa_s\right]\leq 1,\quad for~every~Y \in A\right\}.
\end{displaymath}

We recall that the sets $\mathcal Z$ and $\mathcal Z'$ are defined in \eqref{defZ} and \eqref{defZ'}, respectively. 

\begin{Lemma}\label{finitenessOverZ'}
Under the conditions of Theorem~\ref{mainTheorem}, we have
\begin{equation}\label{bipolarZ'}
(\mathcal Z')^{oo} = \mathcal Y(1)
\end{equation}
and
\begin{equation}\label{2145}
\tilde w(y) = \inf\limits_{Y\in\mathcal Z'}\mathbb E\left[\int_0^TV(t, yY_s)d\kappa_s \right]<\infty, \quad y>0.
\end{equation}
\end{Lemma}
\begin{proof}
It follows from 
 \cite[Theorem 5.2]{DS97} 
 that $\mathcal Z'$ is dense in $\mathcal Z$ 
 in $\mathbf  L^0$. 
It follows from Fatou's lemma that $$(\mathcal Z')^{o} = \mathcal Z^{o}.$$
Therefore, \cite[Lemma 4.2 and Proposition 4.4]{MostovyiNec} imply \eqref{bipolarZ'}.  

One can show that $\mathcal Z'$ is closed under countable convex combinations, where the martingale property follows from the monotone convergence theorem \and the c\`adl\`ag structure of the limit is guaranteed by \cite[Theorem VI.18]{DelMey82}, see also \cite[ Proposition 5.1]{KLPO14} for more details in similar settings. 
Now, \eqref{2145} follows (up to a notational change) from \cite[Theorem 3.3]{MostovyiNec}. This completes the proof of the lemma.
\end{proof}
\begin{Lemma}\label{6-17-1} Under the conditions of Theorem~\ref{mainTheorem}, 
for every $(x,q)\in\mathcal K$ and $(y,r)\in \mathcal L$, we have
$$u(x,q) \leq v(y,r) + xy +\int_0^Tr_sq_sdK_s.$$
\end{Lemma} 
\begin{proof}
Fix an arbitrary $(x,q)\in\cl\mathcal K$, $c\in\mathcal A(x,q)$ as well as $(y,r) \in \mathcal L$, $Y\in\mathcal Y(y,r)$. Using Proposition
\ref{prop1} and \eqref{defV}, we get
\begin{displaymath}
\begin{array}{rcl}
\mathbb E\left[\int_0^T U (t, c_s)d\kappa_s \right] &\leq &\mathbb E\left[\int_0^T U(t, c_s)d\kappa_s \right] + xy + \int_0^Tr_sq_sdK_s - \mathbb E\left[\int_0^T c_sY_sd\kappa_s \right]\\
&\leq &\mathbb E\left[\int_0^T V(t, Y_s)d\kappa_s \right] + xy + \int_0^Tr_sq_sdK_s.  \\
\end{array}
\end{displaymath}
This implies the assertion of the lemma.
\end{proof}
For every $(x,q)$ in $\mathcal K$, we define
\begin{equation}\label{auxiliarySets}
\begin{array}{rcl}
\mathcal B(x,q)&{:=}& \left\{(y,r)\in \mathcal L: ~xy + \int_0^Tr_sq_sdK_s \leq 1 \right\},\\
\mathcal D(x,q)&{:=}& \bigcup\limits_{(y,r)\in \mathcal B(x,q)}\mathcal Y(y,r).\\
\end{array}
\end{equation}
The subsequent lemma established boundedness of $\mathcal B(x,q)$ for $(x,q)$ in $\mathring{\mathcal K}$ in $\mathbb R\times \mathbb L^1(dK)$.
\begin{Lemma}
\label{7-17-1}
Under the conditions of Theorem~\ref{mainTheorem}, for every $(x,q)\in\mathring{\mathcal K}$, 
$\mathcal B(x,q)$ is bounded
 in $\mathbb R\times \mathbb L^1(dK)$.
\end{Lemma}
\begin{proof}
Fix an $(x,q)\in\mathring{\mathcal K}$. Then there exists $\varepsilon > 0$, such that
for every
\begin{equation}\label{7-20-1}
|x'|\leq\varepsilon\quad and \quad||q'||_{\mathbb L^{\infty}} \leq\varepsilon,
\end{equation}
 we have
$(x + x', q+q')\in{\mathring{\mathcal K}}.$
Let us fix an arbitrary $(y,r)\in \mathcal B(x,q)$. Then for every $(x',q')$ satisfying (\ref{7-20-1}), by the definitions of $\mathcal L$ and $\mathcal B(x,q)$, respectively,~we~get 
\begin{displaymath}
\begin{array}{rcl}
xy + \int_0^Tq_sr_sdK_s + x'y + \int_0^Tq'_sr_sdK_s &\geq& 0, \\
xy + \int_0^Tq_sr_sdK_s  &\leq& 1, \\
\end{array}
\end{displaymath}
which implies that
\begin{equation}\label{7-17-2}
-x'y - \int_0^Tq'_sr_sdK_s \leq xy + \int_0^Tq_sr_sdK_s  \leq 1.
\end{equation}
Taking 
\begin{displaymath}
x' = -\varepsilon\quad{and}\quad q'\equiv 0,
\end{displaymath}
we deduce from (\ref{7-17-2}) that $y\leq \frac{1}{\varepsilon}$. Also, by the definition of $\mathcal L$ and Lemma \ref{lemma1}, item $(i)$,  $y\geq 0$. 
In turn, setting
\begin{displaymath}
x' = 0\quad{and}\quad q' = -\varepsilon 1_{\{r\geq 0\}} + \varepsilon 1_{\{r<0\}},
\end{displaymath}
we obtain from  (\ref{7-17-2}) that $||r||_{\mathbb L^1}\leq \frac{1}{\varepsilon}$. This completes the proof of the lemma.
\end{proof}
\begin{Lemma}\label{6-12-1}
Let the conditions of Theorem~\ref{mainTheorem} hold and $(x,q)$ be an arbitrary element of $\mathring{\mathcal K}$.
Then, we have:
\newline
$(i)$ $\mathcal A(x,q)$ contains a strictly positive process.
\newline
$(ii)$ The constant $\bar y(x,q)$ given by  
\begin{equation}\label{defBary}
\bar y(x,q) {:=} \frac{1}{|x| + ||q||_{\mathbb L^\infty(dK)}\sup\limits_{\mathbb Q\in\mathcal M}\mathbb {E^Q}\left[\int_0^T|e_s|d\kappa_s \right]},
\end{equation}
takes values in $(0,\infty)$ and satisfies
\begin{equation}\label{6-18-1}
\bar y(x,q)\mathcal Z'\subseteq \mathcal D(x,q).
\end{equation}
In particular, for every $z>0$, $z\mathcal D(x,q)$ contains a strictly positive process $Y$ such that 
\begin{equation}\label{6-16-2}
\mathbb E\left[\int_0^TV(s, Y_s)d\kappa_s \right] <\infty.
\end{equation}
\end{Lemma}
\begin{proof}
In order to show $(i)$, we observe that the existence of a positive process in $\mathcal A(x,q)$ follows from the fact that $(x-\delta, q)\in\mathcal K$ for a sufficiently small $\delta.$
Now the constant-valued consumption $\delta/A>0$, where $A$ is the constant that dominates the terminal value of the stochastic clock $\kappa$ in \eqref{finClock}, is in $\mathcal A(x,q)$.

In order to prove $(ii)$, 
let us consider $\bar y(x,q)$ given by \eqref{defBary}.
It follows from Lemma \ref{lemma1}, item $(iii)$, that $\bar y (x,q)\in (0,\infty)$. 
For this $\bar y(x,q)$, using 
Corollary \ref{cor1}, one can show~(\ref{6-18-1}). This and Lemma~\ref{finitenessOverZ'}  (note that finiteness of $\tilde w$ follows directly from \eqref{finValue})
imply that for every $z>0$, there exists a positive $Y\in z\mathcal D(x,q)$, such that (\ref{6-16-2}) holds.

\end{proof}
\begin{Lemma}\label{lemFinu}
Under the conditions of Theorem~\ref{mainTheorem}, for every $(x,q)\in\mathring{\mathcal K}$ we have
\begin{displaymath}
-\infty < u(x,q) < \infty.
\end{displaymath}
and $u<\infty$ on $\mathbb R\times \mathbb L^\infty(dK)$.
\end{Lemma}
\begin{proof}
Let us fix an arbitrary $(x,q)\in\mathring{\mathcal K}$.  Since $\mathring{\mathcal K}$ is an open convex cone, there exists $\lambda\in(0,1)$, $(x_1, q_1)\in\mathring{\mathcal K}$, and $x_2>0$, such that
 $$(x,q) =\lambda (x_1, q_1) + (1- \lambda)(x_2,0).$$
Note that $(x_2, 0)\in\mathring{\mathcal K}$ by Lemma \ref{lemma1}. By \eqref{finValue}, there $c\in\mathcal A(x_2,0)$, such that 
\begin{equation}\label{2141}
\mathbb E\left[\int_0^TU\left(t, (1 - \lambda) c_t\right)d\kappa_t  \right] > -\infty.
\end{equation}
As $\mathcal A(x_1,q_1)\neq\emptyset$ (see Remark \ref{remNonEmpty}), there exists $\tilde c\in\mathcal A(x_1,q_1)$. As $U(t, \cdot)$ is nondecreasing, we get
\begin{displaymath}
u(x,q) \geq \mathbb E\left[\int_0^TU(t, \lambda \tilde c_t + (1 - \lambda) c_t)d\kappa_t  \right] 
\geq \mathbb E\left[\int_0^TU(t, (1 - \lambda) c_t)d\kappa_t  \right] 
> -\infty,
\end{displaymath}
where the last inequality follows from \eqref{2141}. This implies finiteness of $u$ on $\mathring{\mathcal K}$ from below. 

In order to show finiteness from above, let us fix a process $c\in\mathcal A(x,q)$, such~that 
\begin{equation}\nonumber
\mathbb E\left[\int_0^TU(t, c_t)d\kappa_t  \right] 
> -\infty.
\end{equation}
By Lemma \ref{finitenessOverZ'}, there exists $Y\in\mathcal Z'$, such that 
\begin{equation}\label{4103}
\mathbb E\left[\int_0^TV(t,Y_t)d\kappa_t \right]<\infty.
\end{equation}
It follows from Lemma \ref{6-16-1} that $Y\in\mathcal Y(1,\rho)$ for some $(1,\rho)\in \mathcal L$. Therefore, by Proposition \ref{prop1}, we get
\begin{equation}\label{4101}
\begin{array}{rcl}
\mathbb E\left[\int_0^T U (s, c_s)d\kappa_s \right]& \leq& \mathbb E\left[\int_0^T U(s,c_s)d\kappa_s \right] + x + \int_0^T\rho_sq_sdK_s - \mathbb E\left[\int_0^T c_sY_sd\kappa_s \right]\\
&\leq& \mathbb E\left[\int_0^T V(s,Y_s)d\kappa_s \right] + x + \int_0^T\rho_sq_sdK_s.  \\
\end{array}
\end{equation}
As $Y$ satisfies (\ref{4103}), we conclude  that $u(x,q) < \infty$. 
Moreover, for $(x,q)\in\mathcal K$, as $\mathcal A(x,q)\neq\emptyset$ by Remark \ref{remNonEmpty}, every $c\in\mathcal A(x,q)$ satisfies \eqref{4101} (with the same $Y$). 
This implies that $u<\infty$ on $\mathcal K$ and therefore, by \eqref{3313}, on $\mathbb R\times \mathbb L^\infty(dK)$. This completes the proof of the lemma.
\end{proof}
We recall that, for every $(x,q)\in\mathcal K$, $\mathcal D(x,q)$ is defined in \eqref{auxiliarySets}. Let $cl \mathcal D(x,q)$ denote the closure of $\mathcal D(x,q)$ in $\mathbb L^0(d\kappa \times \mathbb P).$ The following lemma proves a delicate point that, for  $(x,q)\in\mathring{\mathcal K}$, by passing from $\mathcal D(x,q)$ to $cl\mathcal D(x,q)$, we do not change the auxiliary dual value function. 
\begin{Lemma}\label{auxiliaryLemma}
Let the conditions of Theorem~\ref{mainTheorem} hold and $(x,q)\in\mathring{\mathcal K}$. Then, for every $z>0$, we have
\begin{equation}\label{2143}
-\infty<\inf\limits_{Y \in\cl \mathcal D(x,q)}\mathbb E\left[\int_0^TV(s, zY_s)d\kappa_s \right] =
\inf\limits_{Y \in \mathcal D(x,q)}\mathbb E\left[\int_0^TV(s, zY_s)d\kappa_s \right]<\infty.
\end{equation}
\end{Lemma}
\begin{proof}
Finiteness from above follows from Lemma \ref{6-12-1}. To show finiteness of both infima in \eqref{2143} from below, by Lemma~\ref{lemFinu} we deduce
the existence of $c\in\mathcal A(x,q)$, such that
\begin{equation}\label{6-18-2}
\mathbb E\left[\int_0^TU(s, c_s)d\kappa_s \right] > -\infty.
\end{equation}
Let $Y\in{\cl \mathcal D(x,q)}$ and let $Y^n\in\mathcal Y(y^n,r^n)$, $n\geq 1$, be a sequence in $\mathcal D(x,q)$ that converges to $Y$ in $\mathbb L^0(d\kappa\times\mathbb P)$. By Fatou's lemma, Proposition~\ref{prop1}, and the definition of the set $\mathcal B(x,q)$ in \eqref{auxiliarySets}, we get
\begin{displaymath}
\mathbb E\left[\int_0^TY_sc_sd\kappa_s \right] \leq 
\liminf\limits_{n\to\infty}\mathbb E\left[\int_0^TY^n_sc_sd\kappa_s \right] \leq
\sup\limits_{n\geq 1}\left(xy^n + \int_0^Tr^n_sq_sdK_s\right) \leq 1. 
\end{displaymath}
Therefore, we obtain
\begin{displaymath}
\begin{array}{rcl}
\mathbb E\left[ \int_0^TU(s, c_s)d\kappa_s\right]&\leq& \mathbb E\left[ \int_0^TU(s, c_s)d\kappa_s\right] + 1 - \mathbb E\left[\int_0^TY_sc_sd\kappa_s \right] \\
&\leq& \mathbb E\left[ \int_0^TV(s, Y_s)d\kappa_s\right] + 1,\\
\end{array}
\end{displaymath}
which together with (\ref{6-18-2}) implies finiteness of both infima in \eqref{2143} from below.

Let us show equality of two infima in \eqref{2143}. It follows from Lemma 
\ref{6-12-1} that for every $z> 0$ there exists a process $Y\in z\mathcal D(x,q)$, such that
$$\mathbb E\left[\int_0^TV(s, Y_s)d\kappa_s \right] <\infty.$$
Let us fix $z>0$ and let $\bar Y\in\cl\mathcal D(x,q)$. Also, let $(Y^n)_{n\in\mathbb N}$ be a sequence in $\mathcal D(x,q)$ that converges to $\bar Y$ $(d\kappa\times\mathbb P)$-a.e. 
Let us fix $\delta >0$, then by Lemma \ref{finitenessOverZ'}, there exists $Z'\in\mathcal Z'$, such that 
$$\mathbb E\left[\int_0^TV(t, \delta\bar y(x,q) Z'_t)d\kappa_t \right]<\infty,$$
where $\bar y(x,q)$ is defined in \eqref{defBary}. Note that $\bar y(x,q) Z'\in \mathcal D(x,q)$ by Lemma \ref{6-12-1} (see \eqref{6-18-1}). 
Therefore, using Fatou's lemma and monotonicity of $V$ in the spatial variable, we obtain
\begin{displaymath}
\begin{array}{rcl}
\inf\limits_{Y \in \mathcal D(x,q)}\mathbb E\left[\int_0^TV(t, (z+\delta)Y_s)d\kappa_s \right]&\leq&
\limsup\limits_{n\to\infty}\mathbb E\left[\int_0^TV(t, zY^n_s+\delta \bar y(x,q)Z'_s)d\kappa_s \right]\\ 
&\leq  &\mathbb E\left[\int_0^TV(t, z\bar Y_s + \delta\bar y(x,q) Z'_s)d\kappa_s \right] \\
&\leq &\mathbb E\left[\int_0^TV(t, z\bar Y_s )d\kappa_s \right].\\
\end{array}
\end{displaymath}
Taking the infimum over $Y\in\cl\mathcal D(x,y)$, we deduce that
\begin{equation}\label{21410}
 \inf\limits_{Y \in \mathcal D(x,q)}\mathbb E\left[\int_0^TV(t, (z+\delta)Y_s)d\kappa_s \right] \leq \inf\limits_{Y\in\cl\mathcal D(x,y)}\mathbb E\left[\int_0^TV(t, zY_s )d\kappa_s \right].
\end{equation}
Let us consider 
$$\phi(z){:=} \inf\limits_{Y \in \mathcal D(x,q)}\mathbb E\left[\int_0^TV(t, zY_s)d\kappa_s \right], \quad z>0.$$
By the first part of the proof (finiteness of both infima), $\phi$ is finite-valued on $(0,\infty)$.  Convexity of $V$ in the spatial variable implies that $\phi$ is also convex.  Therefore, $\phi$ is continuous. As \eqref{21410} holds for every $\delta>0$, by taking the limit as $\delta \downarrow 0$ in \eqref{21410}, we conclude that both infima in \eqref{2143} are equal. This completes the proof of the lemma.

\end{proof}

Let us define
\begin{equation}\label{defE}
\mathcal E {:=} \{(y,r)\in \mathcal L: ~v(y,r) < \infty \}.
\end{equation}
\begin{Lemma}\label{lemFinv}
Under the conditions of Theorem~\ref{mainTheorem}, for every $(y,r)\in \mathcal L$, we have
$$v(y,r) > -\infty.$$
Therefore, $v>-\infty$ on $\mathbb R\times\mathbb L^1(dK)$. The set 
 $\mathcal E$ is a nonempty convex subset of $\mathcal L$, whose closure in $\mathbb R\times\mathbb L^1(dK)$ equals to $\mathcal L$, and such that
\begin{equation}\label{6-17-2}
\mathcal E= \bigcup\limits_{\lambda \geq 1 }\lambda \mathcal E.
\end{equation}

\end{Lemma}
\begin{proof}
Let us  fix $(y,r)\in \mathcal L$, then finiteness of $v(y,r)$ from below follows from  \eqref{finValue} and Lemma~\ref{6-17-1}.
To establish the properties of $\mathcal E$, we observe that the convexity of $\mathcal E$  and (\ref{6-17-2}) follow from convexity and monotonicity of $V$, respectively.

In remains to show that the closure of $\mathcal E$ in $\mathbb R\times\mathbb L^1(dK)$ contains the origin. 
In \eqref{auxiliarySets}, let us consider $(x,q) = (1,0)\in\mathcal K$. In this case, we have 
 $$\mathcal D(1,0) = \bigcup\limits_{(y,r)\in\mathcal L: y\leq 1}\mathcal Y(y,r)\subseteq \mathcal Y(1),$$
where the last inclusion follows from the  very  definition of $\mathcal Y(y,r)$'s in \eqref{defY}.
As, by \eqref{oldY}, $\mathcal Y(1)$ is closed in $\mathbb L^0(d\kappa\times\mathbb P)$  and $\mathcal D(1,0)\subseteq \mathcal Y(1)$, we deduce that 
\begin{equation}\label{2146}
\cl \mathcal D(1,0)\subseteq \mathcal Y(1).
\end{equation}
 By Lemma \ref{6-12-1}, $\mathcal Z'\subset \mathcal D(1,0)$, as $\bar y(1,0) = 1$. Therefore, by the bipolar theorem of Brannath and Schachermayer, \cite[Theorem 1.3]{BranSchach}, we get
\begin{equation}\label{2147}
(\mathcal Z')^{oo} \subseteq\cl\mathcal D(1,0).
\end{equation}
On the other hand, Lemma \ref{finitenessOverZ'} asserts that 
\begin{equation}\label{2148}
(\mathcal Z')^{oo} = \mathcal Y(1).
\end{equation}
Combining \eqref{2146}, \eqref{2147}, and \eqref{2148}, we conclude
\begin{equation}\nonumber
\cl \mathcal D(1,0)= \mathcal Y(1).
\end{equation}
Therefore, the sets $\cl\mathcal D(1,0)= \mathcal Y(1)$ and $\mathcal A(1,0)$ satisfy the  precise technical  assumptions of \cite[Theorem 3.2]{MostovyiNec}, which, for every $x>0$, grants  the existence of $\widehat c(x)\in\mathcal A(x,0)$, the unique maximizer to $w(x)$, where $w$ is defined in \eqref{defw}.
For every $x>0$, we set
$$
Y_{\cdot}(x){:=} U'({\cdot}, \widehat c_{\cdot}(x)), \quad 
(d\kappa\times\mathbb P)-a.e.
$$
By \cite[Theorem 3.2]{MostovyiNec}, for every $x>0$,  $Y(x)$  satisfies
$$Y(x) \in w'(x)\cl\mathcal D(1,0)\quad and\quad \mathbb E\left[ \int_0^TV(t,  Y_t(x))d\kappa_t\right]<\infty,\quad x>0,$$
where by \cite[Theorem 3.2]{MostovyiNec}, $w$ is a strictly concave, differentiable function on $(0,\infty)$ that satisfies the Inada conditions. 
Therefore, as $w'(x)$ can be arbitrary close to $0$ (by taking $x$ large enough an by using the Inada conditions) and by Lemmas \ref{6-12-1} and \ref{auxiliaryLemma}, we conclude that the closure of $\mathcal E$ in $\mathbb R\times \mathbb L^1(dK)$ contains origin. 

In order to prove that the closure of $\mathcal E$ in $\mathbb R\times \mathbb L^1(dK)$ equals to $\mathcal L$, let $(y,r)\in\mathcal L\backslash (0, 0)$ be fixed. Let us take $\varepsilon>0$. We want to find $(\tilde y, \tilde r)$, such that 
\begin{equation}
\label{4171}
|\tilde y - y| + || \tilde r- r ||_{\mathbb L^1(dK)}<\varepsilon,
\end{equation}
and
\begin{equation}\label{4172}
(\tilde y, \tilde r)\in\mathcal E.
\end{equation}
As the closure of $\mathcal E$ in $\mathbb R\times \mathbb L^1(dK)$ contains origin, we can pick $(y^0, r^0)\in\mathcal E$, such that 
\begin{equation}\label{4173}
|y^0| + ||r^0||_{\mathbb L^1(dK)} \leq \varepsilon/3
\end{equation}
and $Y\in\mathcal Y(y^0,r^0)$, such that 
\begin{equation}\label{4174}
\mathbb E\left[\int_0^T V(t, Y_t)d\kappa_t \right]<\infty.
\end{equation}
Let us fix $\alpha >1$, such that 
\begin{equation}\label{4177}
\frac{|y| + ||r||_{\mathbb L^1(dK)}}{\alpha} \leq \varepsilon/3
\end{equation}
 and set $\varepsilon' {:=} \frac{1}{\alpha}\in(0,1)$. By \eqref{6-17-2}, $(\alpha y^0, \alpha r^0)\in\mathcal E.$ Let 
\begin{equation}\nonumber
\tilde y {:=}  (1-\varepsilon')y + \varepsilon'\alpha y^0,\quad \tilde r {:=} (1-\varepsilon')r + \varepsilon'\alpha r^0.
\end{equation} 
Then
\begin{displaymath}
\begin{array}{rcl}
|y - \tilde y| + || r - \tilde r||_{\mathbb L^1(dK)} &= & \varepsilon' \alpha |\frac{y}{\alpha} - y^0| + \varepsilon'\alpha || \frac{r}{\alpha} - r^0||_{\mathbb L^1(dK)}\\
&\leq & \frac{|y| + ||r||_{\mathbb L^1(dK)}}{\alpha} + |y^0| + ||r^0||_{\mathbb L^1(dK)}
\\
&\leq & \frac{2\varepsilon}{3},
\end{array}
\end{displaymath}
where in the last inequality we have used \eqref{4173} and \eqref{4177}. Thus $(\tilde y, \tilde r)$ satisfies \eqref{4171}. Further, as $0\in\mathcal Y(y,r)$, by convexity of $\mathcal L$ and using Proposition \ref{prop1}, we get
\begin{displaymath}
Y=(1-\varepsilon') 0 + \varepsilon' \alpha Y \in\mathcal Y(\tilde y, \tilde r),
\end{displaymath}
which by \eqref{4174} implies \eqref{4172}. This completes the proof of the lemma.
\end{proof}

\subsection{Existence and uniqueness of solutions to \eqref{primalProblem} and \eqref{dualProblem}; semicontinuity and biconjugacy of $u$ and $v$}

\begin{Lemma}\label{existenceUniqueness}
Under the conditions of Theorem~\ref{mainTheorem},
the value function $v$ is convex, proper, and lower semicontinuous with respect to the  topology of $\mathbb R\times \mathbb L^1(dK)$. For every $(y,r)\in\mathcal E$, there exists a unique solution to 
(\ref{dualProblem}). Likewise, $u$ is concave, proper, and upper semicontinuous  with respect to  the  strong  topology of $\mathbb R\times \mathbb L^\infty(dK)$. For every $(x,q)\in\{u>-\infty\}$ there exists a unique solution to~(\ref{primalProblem}).
\end{Lemma}
\begin{proof}
Let $(y^n, r^n)_{n\in\mathbb N}$ be a sequence in $\mathcal L$ that converges to $(y,r)$ in $\mathbb R\times\mathbb L^1(dK)$.  
Passing if necessary to a subsequence, we will assume that 
\begin{equation}\label{2151}
\lim\limits_{n\to\infty} v(y^n,r^n)= \liminf\limits_{n\to\infty}v(y^n, r^n).
\end{equation}
Let $Y^n\in\mathcal Y(y^n, r^n)$, $n\in\mathbb N$, be such that 
\begin{equation}\label{2152}
\mathbb E\left[ \int_0^TV(t,Y^n_t)d\kappa_t\right]\leq v(y^n, r^n) +\frac{1}{n}, \quad n\in\mathbb N.
\end{equation}
By passing to convex combinations and applying Komlos'-type lemma, see e.g. \cite[Lemma A1.1]{DS}, we may suppose that $\widetilde Y^n\in\conv\left(Y^n, Y^{n+1},\dots\right)$, $n\in\mathbb N$, converges $(d\kappa\times\mathbb P)$-a.e. to some $\widehat Y$. 

For every $(x,q)\in\mathcal K$ and $c\in\mathcal A(x,q)$, by Fatou's lemma, we have
\begin{displaymath}
\begin{array}{c}
\mathbb E\left[ \int_0^T c_t\widehat Y_td\kappa_t\right]\leq \liminf\limits_{n\to\infty}\mathbb E\left[ \int_0^T c_t \widetilde Y^n_td\kappa_t\right]\leq 
xy + \int_0^Tq_sr_sdK_s.
\end{array}
\end{displaymath} 
Therefore, by Proposition \ref{prop1}, $\widehat Y\in\mathcal Y(y,r)$. With $\bar y {:=} \sup\limits_{n\geq 1}y^n$, we have $(\widetilde Y^n)_{n\in\mathbb N}\subseteq \mathcal Y(\bar y).$ Therefore, by \cite[Lemma 3.5]{MostovyiNec}, we deduce that 
$V^{-}(t,\widetilde Y^n_t)$, $n\in\mathbb N$, is a uniformly integrable sequence.  
Combining uniform integrability with the convexity of $V$ in the spatial variable, we get
\begin{equation}\label{4141}
\begin{array}{c}
v(y,r)\leq \mathbb E\left[ \int_0^TV(t,\widehat Y_t)d\kappa_t\right]\leq 
\liminf\limits_{n\to\infty}
\mathbb E\left[ \int_0^TV(t, \widetilde Y^n_t)d\kappa_t\right]
\\
\leq 
\liminf\limits_{n\to\infty}
\mathbb E\left[ \int_0^TV(t,  Y^n_t)d\kappa_t\right] = \liminf\limits_{n\to\infty}v(y^n,r^n),
\end{array}
\end{equation}
where in the last equality we have used \eqref{2151} and \eqref{2152}. Since $(y^n, q^n)$ was an arbitrary sequence that converges to $(y,r)$, lower semicontinuity of $v$ in strong  topology of $\mathbb R\times \mathbb L^1(dK)$ follows. Since $\mathcal L$ is closed and $v = \infty$ outside of $\mathcal L$, we deduce that $v$ is lower semicontinuous on $\mathbb R\times \mathbb L^1(dK)$. The  function $v$ is proper by Lemma \ref{lemFinv}. Note that \eqref{4141} also implies that $\mathcal E$ defined in \eqref{defE} is $\mathbb R\times\mathbb L^1(dK)$-norm closed. For $(y,r)\in\mathcal E$, by taking $(y^n, r^n)= (y,r)$, $n\in\mathbb N$, we deduce the existence of a minimizer to \eqref{dualProblem}. Strict convexity of $V$ results in the uniqueness of the minimizer to \eqref{dualProblem}. Convexity of $v$ follows.  Upper semicontinuity of $u$  with respect to the norm-topology of   $\mathbb R\times \mathbb L^{\infty}(dK)$ can be proven similarly, first proving semi-continuity on 
 $\mathcal K$  by a Fatou-type argument, then using the closedness  of $\mathcal{K}$ and the definition of $u$ outside it. 
\end{proof}
\begin{Corollary}
Under the conditions of Theorem~\ref{mainTheorem},
$-u$ and $v$ are  also  lower semicontinuous with respect to the weak topologies $\sigma(\mathbb R\times\mathbb L^{\infty}(dK), (\mathbb R\times\mathbb L^{\infty})^*(dK))$ and $\sigma(\mathbb R\times\mathbb L^1(dK),\mathbb R\times\mathbb L^{\infty}(dK))$, respectively.
\end{Corollary}
\begin{proof}
The assertions of the corollary is a consequence of \cite[Proposition 2.2.10]{BarbuPrec}, see also \cite[Corollary I.2.2]{EkelandTemam}.
\end{proof}

We recall that ${\cl \mathcal D(x,q)}$ denotes the closure of $\mathcal D(x,q)$ in $\mathbb L^0(d\kappa\times\mathbb P)$. 
\begin{Lemma}\label{5181}
Under the conditions of Theorem \ref{mainTheorem}, for every $(x,q)$ in $\mathring{\mathcal K}$, and a nonnegative optional process $c$, we have
$$
c\in\mathcal A(x,q)\quad{if~and~only~if}\quad \mathbb E\left[\int_0^T c_sY_sd\kappa_s \right]\leq 1\quad {for~every~}Y\in{\cl \mathcal D(x,q)}.
$$
\end{Lemma}
\begin{proof}
Let $(x,q)$ in $\mathring{\mathcal K}$, $c$ is a nonnegative optional process such that 
\begin{equation}\label{5191}
\mathbb E\left[ \int_0^T c_sY_sd\kappa_s \right]\leq 1\quad {for~every}\quad Y\in{\cl \mathcal D(x,q)}.
\end{equation}
Consider arbitrary $Z\in\mathcal Z'$ and $\Lambda\in\Upsilon$. Let the corresponding $r$ be given by \eqref{6-25-1} and we set $$y'{:=} x + \int_0^Tq_sr_sdK_s.$$

If $y' = 0$, then $y\Lambda Z\in\mathcal D(x,q)$ for every $y>0$. Thus,  by \eqref{5191}, we obtain that $\mathbb E\left[\int_0^T c_sy\Lambda_sZ_sd\kappa_s \right]\leq 1$. Taking the limit as $y\to\infty$, we get
\begin{equation}\label{5183}\mathbb E\left[\int_0^T c_s\Lambda_sZ_sd\kappa_s \right] = 0=  x + \int_0^Tq_sr_sdK_s= x + \mathbb E\left[\int_0^Tq_se_s\Lambda_sZ_sd\kappa_s \right],
\end{equation}
where in the last equality we have used \eqref{6-25-1}. 

If $y'>0$, 
then $\frac{1}{y'}\Lambda Z\in\mathcal D(x,q)$ and thus by \eqref{5191}, we obtain $$\mathbb E\left[ \int_0^T c_s\frac{1}{y'}\Lambda_sZ_sd\kappa_s\right]\leq 1 = \frac{ x + \int_0^Tq_sr_sdK_s}{y'} = \frac{1}{y'}\left( x + \mathbb E\left[\int_0^T q_se_s\Lambda_sZ_sd\kappa_s\right]\right),$$
where in the last equality, we have used \eqref{6-25-1} again. Consequently, we deduce
$$\mathbb E\left[ \int_0^T c_s\Lambda_sZ_sd\kappa_s\right]\leq  x + \mathbb E\left[\int_0^T q_se_s\Lambda_sZ_sd\kappa_s\right],$$
which together with \eqref{5183}, by Lemma \ref{9-5-1}, imply that $c\in\mathcal A(x,q)$.

Conversely, let $(x,q)\in\mathring{\mathcal K}$, $c\in\mathcal A(x,q)$ and $Y\in {\rm cl}\mathcal D(x,q)$. Then there exists a sequence $Y^n\in\mathcal Y(y^n, r^n)$ convergent to $Y$, $(d\kappa\times\mathbb P)$-a.e., where $(y^n,r^n)\in\mathcal B(x,q)$. As, 
$$\mathbb E\left[\int_0^T c_sY^n_sd\kappa_s \right]\leq 1,\quad n\in\mathbb N,$$
by Fatou's lemma, we get
$$\mathbb E\left[\int_0^T c_sY_sd\kappa_s \right]\leq \liminf\limits_{n\to\infty}\mathbb E\left[\int_0^T c_sY^n_sd\kappa_s \right]\leq 1.$$
This completes the proof of the lemma.
\end{proof}
\begin{Lemma}\label{conjugacy}
Under the conditions of Theorem \ref{mainTheorem}, for every $(x,q)$ in $\mathring{\mathcal K}$, we have
\begin{equation}\label{conjugacyOneDirection}
u(x,q) = \inf\limits_{(y,r)\in \mathcal L}\left(v(y,r) + xy + \int_0^Tr_sq_sdK_s\right).
\end{equation}
\end{Lemma}
\begin{proof} Let us fix $(x,q)\in\mathring{\mathcal K}$. 
By Lemma \ref{6-12-1}, $\mathcal A(x,q)$ and ${\cl \mathcal D(x,q)}$ contain strictly positive elements. 
Therefore, using Lemma \ref{5181}
we deduce that the sets $\mathcal A(x,q)$ and ${\cl \mathcal D(x,q)}$ satisfy the assumptions 
of \cite[Theorem 3.2]{MostovyiNec}. From this theorem, Lemma~\ref{auxiliaryLemma}, and the definition of the set $\mathcal B(x,q)$, we get
\begin{equation}\nonumber
\begin{array}{rcl}
u(x,q) &=& \inf\limits_{z>0}\left(\inf\limits_{Y\in{\cl \mathcal D(x,q)}}\mathbb E\left[ \int_0^TV(t, zY_s)d\kappa_s\right] + z \right) \\
 &=& \inf\limits_{z>0}\left(\inf\limits_{Y\in\mathcal D(x,q)}\mathbb E\left[ \int_0^TV(t, zY_s)d\kappa_s\right] + z \right) \\
&=& \inf\limits_{z>0}\left(\inf\limits_{(y,r)\in z\mathcal B(x,q)}v(y,r) + z \right) \\
&\geq& \inf\limits_{(y,r)\in \mathcal L}\left(v(y,r) + xy + \int_0^Tq_sr_sdK_s \right).\\
\end{array}
\end{equation}
Combining this with the conclusion of Lemma~\ref{6-17-1}, we deduce that \eqref{conjugacyOneDirection} holds for every $(x,q)\in\mathring{\mathcal K}$.

%
\end{proof}

Before proving the biconjugacy relations of item $(iii)$, Theorem \ref{mainTheorem}, we need a preliminary lemma. 
Essentially following the notations in \cite{EkelandTemam}, we define 
\begin{equation}\label{v*}
v^*(x,q){:=} \inf\limits_{(y,r)\in\mathbb R\times \mathbb L^1(dK)}\left(v(y,r) + xy + \int_0^Tq_sr_sdK_s \right),\quad (x,q)\in\mathbb R\times \mathbb L^{\infty}(dK).
\end{equation}
\begin{equation}\label{v**}
v^{**}(y,r){:=} \sup\limits_{(x,q)\in\mathbb R\times \mathbb L^\infty(dK)}\left(v^*(x,q) - xy - \int_0^Tq_sr_sdK_s \right),\quad (y,r)\in\mathbb R\times \mathbb L^1(dK).
\end{equation}
\begin{Remark}In \cite{EkelandTemam}, conjugate convex functions are considered on general spaces $V$ and $V^*$ supplied with $\sigma(V,V^*)$ and $\sigma(V^*, V)$ topologies, which in our case are $V= \mathbb R\times\mathbb L^1(dK)$, $V^* = \mathbb R\times\mathbb L^\infty(dK)$. Thus, the starting point of our analysis is $v$, not $u$. We remind the reader we have already proved that the dual value function $v$ is convex, proper and lower-semicontinuous on the space $V= \mathbb R\times\mathbb L^1(dK)$. 
\end{Remark}
\begin{Lemma}\label{lemv**}
Under the conditions of Theorem \ref{mainTheorem}, we have 
\begin{equation}
\label{2166}
v^{**} = v, 
\end{equation}
\begin{equation}
\label{4151}
v^*(x,q) = -\infty, \quad for~every\quad (x,q)\in\mathbb R\times\mathbb L^\infty(dK)\backslash \mathcal K.
\end{equation}

\end{Lemma}

\begin{proof}
To show \eqref{2166}, we observe that by Lemma \ref{existenceUniqueness}, $v$ is lower semicontinuous in the $\mathbb R\times\mathbb L^1(dK)$-norm topology (and therefore, by \cite[Corollary I.2.2]{EkelandTemam}, also in the weak topology $\sigma(\mathbb R\times\mathbb L^1(dK),\mathbb R\times\mathbb  L^\infty(dK))$). As a result, by \cite[Proposition I.4.1]{EkelandTemam}, we get~\eqref{2166}.

The proof of \eqref{4151} will be done in several steps. 

{\it Step 1.} Let $(x,q)\in\mathbb R\times\mathbb L^\infty(dK)\backslash \mathcal K$.  According to  Proposition \ref{prop1}, there exists $(y,r)\in\mathcal L$, such that
\begin{equation}\nonumber
C{:=} xy + \int_0^T q_sr_sdK_s <0.
\end{equation}
Therefore, as $\mathcal L$ is a cone, for every $a>0$, $(ay, ar)\in\mathcal L$, and we have
\begin{equation}\label{4153}
xay + \int_0^Tq_sar_sdK_s = aC<0.
\end{equation}
Note that $0\in\mathcal Y(ay, ar)$, $a>0$.

{\it Step 2.} Let us consider $Z\in\mathcal Z'$, such that 
\begin{equation}\label{4154}
\mathbb E\left[\int_0^TV\left(t,\tfrac{1}{2}Z_t\right)d\kappa_t\right]<\infty.
\end{equation}
The existence of such a $Z$ is granted by Lemma \ref{finitenessOverZ'}.
Further, by Corollary \ref{cor1}, there exists $\rho\in\mathbb L^1(dK)$, such that $(1, \rho)\in\mathcal L$, $Z\in\mathcal Y(1,\rho)$,  and
\begin{equation}\nonumber
\int_0^t\rho_sdK_s = \mathbb E\left[\int_0^t Z_se_sd\kappa_s \right],\quad t\in[0,T].
\end{equation}
Let us set 
\begin{equation}\nonumber
D{:=} x + \int_0^Tq_s\rho_sdK_s\in\mathbb R.
\end{equation}

{\it Step 3.} In \eqref{4153}, let us pick 
\begin{equation}\nonumber
a = \frac{|D| + 1}{-C}>0.
\end{equation}
Then, we have 
\begin{equation}\label{4158}
aC + D = -|D| + D -1<0.
\end{equation}

{\it Step 4.} Let us define 
\begin{equation}\label{4159}
Y{:=} \tfrac{1}{2} Z, \quad y'{:=} \tfrac{1}{2}ay + \tfrac{1}{2}, \quad and \quad r'{:=} \tfrac{1}{2}ar + \tfrac{1}{2}\rho.
\end{equation}
Then by \eqref{4154}, we obtain
\begin{equation}\label{41510}
\mathbb E\left[\int_0^TV\left(t,Y_t\right)d\kappa_t\right]<\infty.
\end{equation}
As $Z\in\mathcal Y(1,\rho)$ and $0\in\mathcal Y(ay, ar)$, by convexity of $\mathcal L$ and Proposition \ref{prop1}, we have
\begin{equation}\label{41511}
Y\in\mathcal Y(y',r'),
\end{equation}
where $y'$ and $r'$ are defined in \eqref{4159}. Now, it follows from  \eqref{41510} and \eqref{41511} that $(y', r')\in\mathcal E$ (where $\mathcal E$ is defined in \eqref{defE}). 
Therefore, we obtain
\begin{displaymath}
\begin{array}{c}
2(xy' + \int_0^Tq_sr'_sdK_s) = (ay + 1)x + \int_0^T(ar_s + \rho_s)q_sdK_s \\= a(xy + \int_0^Tq_sr_sdK_s) + x + \int_0^Tq_s\rho_sdK_s = aC + D <0,\\
\end{array}
\end{displaymath}
where the last inequality follows from \eqref{4158}.
To  recapitulate, we have shown the existence of  $(y', r')$, such that
\begin{equation}\label{41512}
(y', r')\in\mathcal E\quad and\quad xy' + \int_0^Tq_sr'_sdK_s<0.
\end{equation}

{\it Step 6.} For $y'$ and $r'$ defined in \eqref{4159}, as $v(y', r') <\infty$ and $xy' + \int_0^Tq_sr'_sdK_s<0$ by \eqref{41512}, from the monotonicity of $V$, we get
$$\infty>v(y',r')\geq v(\lambda y', \lambda r'),\quad \lambda \geq 1.$$
As $\bigcup\limits_{\lambda\geq 1}(\lambda y', \lambda r')\subset \mathcal L$, we conclude via \eqref{41512} that
$$v^*(x,q) \leq \lim\limits_{\lambda \to\infty}\left( v(\lambda y', \lambda r') + \lambda \left(xy' + \int_0^Tq_sr'_sdK_s\right)\right) = -\infty.$$
Therefore, \eqref{4151} holds. This completes the proof of the lemma.
\end{proof}
\begin{Lemma}\label{lemBiconjugacy}
Under the conditions of Theorem \ref{mainTheorem}, we have 
\begin{equation}\label{2161}
v(y,r) = \sup\limits_{(x,q)\in\mathcal K}\left(u(x,q) - xy - \int_0^Tr_sq_sdK_s\right),\quad (y,r)\in\mathcal L,
\end{equation}
\begin{equation}\label{21610}
u(x,q) = \inf\limits_{(y,r)\in\mathcal L}\left(v(y,r) + xy + \int_0^Tq_sr_sdK_s\right),\quad (x,q)\in \mathcal K.
\end{equation}
\end{Lemma}
\begin{proof}

Lemma \ref{conjugacy} and \eqref{3313} imply that on $\mathring{\mathcal K}$, for $v^*$ defined in \eqref{v*}, we have
\begin{equation}\label{2164}
v^{*} = u.
\end{equation}
By Lemma \ref{lemv**}, $v^* = -\infty$ on $\mathbb R\times\mathbb L^\infty(dK)\backslash\mathcal K$. 
From \cite[Definition I.4.1]{EkelandTemam} and Lemma \ref{existenceUniqueness}, respectively, we deduce that both $v^*$ and $u$ are upper semicontinuous in the topology of $\mathbb R\times\mathbb L^\infty(dK)$
. Consequently, from \eqref{2164}, using \cite[Corollary I.2.1]{EkelandTemam},   
we get
\begin{equation}\label{2201}v^* = u,\quad on\quad\mathbb R\times\mathbb L^\infty(dK).
\end{equation}
As a result, with $v^{**}$ being defined in \eqref{v**}, for every $(y,r)\in\mathbb R\times \mathbb L^1(dK)$, we obtain
\begin{equation}\label{2165}
u^{*}(y,r){:=} \sup\limits_{(x,q)\in\mathbb R\times \mathbb L^\infty(dK)}\left(u(x,q) - xy - \int_0^Tq_sr_sdK_s \right)=v^{**}(y,r).
\end{equation}
Therefore, from \eqref{2166} in Lemma \ref{lemv**} and \eqref{2165}, we get
\begin{equation}\nonumber
v = u^{*},\quad on\quad\mathbb R\times\mathbb L^1(dK).
\end{equation}
As a result, applying Lemma \ref{lemv**} again and since $u = -\infty$ outside of $\mathcal K$ by \eqref{3313}, we deduce 
\begin{equation}\nonumber
\begin{array}{rcl}
v(y,r)  &=& \sup\limits_{(x,q)\in\mathbb R\times\mathbb L^\infty(dK)}\left( u(x,q)- xy - \int_0^Tr_sq_sdK_s\right)\\
&=& \sup\limits_{(x,q)\in\mathcal K}\left( u(x,q) - xy - \int_0^Tr_sq_sdK_s\right),\quad(y,r)\in\mathcal L,\\
\end{array}
\end{equation}
Thus, \eqref{2161} holds. 

In turn, from \eqref{2201} using \eqref{3313}, 
we conclude that

\begin{equation}\nonumber
\begin{array}{rcl}
u(x,q) &=& \inf\limits_{(y,r)\in\mathbb R\times \mathbb L^1(dK)}\left(v(y,r) + xy + \int_0^Tq_sr_sdK_s\right),\\
&=& \inf\limits_{(y,r)\in\mathcal L}\left(v(y,r) + xy + \int_0^Tq_sr_sdK_s\right),\quad (x,q)\in\mathcal K,\\
\end{array}
\end{equation}
which proves \eqref{21610} and extends the assertion of Lemma \ref{conjugacy} to the boundary of $\mathring{\mathcal K}$.

\end{proof}

\begin{proof}[Proof of Theorem \ref{mainTheorem}]
The assertions of item $(i)$ follow from Lemmas \ref{lemFinu} and \ref{lemFinv}, item $(ii)$ results from Lemma \ref{existenceUniqueness}, whereas the validity of item $(iii)$ come from Lemma 
\ref{lemBiconjugacy}. This completes the proof of the theorem.
\end{proof}


\section{Subdifferentiability of $u$}\label{secSubdifferentiability}
In order to establish subdifferentiability of $u$, we need to strengthen 
Assumption~\ref{asEnd} and to impose the following condition. 
\begin{Assumption}\label{asACclock} 
There exists an a.s. bounded away from $0$ and $\infty$  
process $\varphi$, such that
\begin{equation}\nonumber
d\kappa(\omega) =\varphi d K, \quad for\quad\mathbb P-a.e.\quad\omega \in\Omega.
\end{equation}

\end{Assumption}


%
%
%
%

Let 
\begin{equation}\label{defPi}
\begin{array}{rcl}
\mathcal P &{:=} &\left\{ \rho:~(1,\rho)\in\mathcal L\right\},\\
\end{array}
\end{equation}
\begin{Remark} 

  $\mathcal P$ defined in \eqref{defPi} needs to be uniformly integrable with respect to the measure $dK$ in order  for the proof of subdifferentiability of $u$ to go through.
 Assumption \ref{asEnd} through Lemma \ref{lemma1} only implies  
 that $\mathcal P$ is $\mathbb{L}^1$ bounded. A stronger condition on the stochastic clock and income stream is, therefore, needed to obtain uniform integrability. 
 

\end{Remark}
The following theorem characterizes subdifferentiability of $u$  over  $\mathring{\mathcal K}$, 
where we are looking for an  $\mathbb R\times\mathbb L^1(dK)$-valued subgradient.   Under our assumptions,  we can find elements of the sub gradient which both  belong to the  effective domain of $v$, $\mathcal E$, and are bounded, i.e. in $\mathbb{R}\times \mathbb{L}^{\infty} (dK)$.
\begin{Theorem}\label{mainTheorem2} 
Let the conditions of Theorem~\ref{mainTheorem} and Assumption \ref{asACclock} hold. Then for every $(x,q)\in\mathring{\mathcal K}$, the subdifferential of $u$ at $(x,q)$ is a nonempty and contains an element of $\mathcal E$, i.e.,
\begin{equation}\label{subdifNonempty}
\partial u(x,q)\cap \mathcal E\neq \emptyset.
\end{equation}
Moreover, for $(x,q)\in\mathring{\mathcal K}$ and $(y,r)\in\mathcal L$, $(y,r)\in\partial u(x,q)$ if and only if the following conditions hold:
\begin{equation}\label{2191}
 |v(y,r)|<\infty,
 \end{equation}
 thus, $(y,r)\in\mathcal E$, 
 \begin{equation}\label{2192}
 \mathbb E\left[\int_0^T\widehat Y_t(y,r) \widehat c_t(x,q)d\kappa_t\right] = xy + \int_0^Tq_sr_sdK_s,
 \end{equation}
 \begin{equation}\label{2193}
 \widehat Y_t(y,r) = U'(t,\widehat c_t(x,q)), \quad (d\kappa\times\mathbb P)- a.e.,
 \end{equation}
 where $\widehat c(x,q)$ and $\widehat Y(y,r)$ are the unique optimizers to \eqref{primalProblem} and \eqref{dualProblem}, respectively.
\end{Theorem}



\subsection{Uniform integrability of $\mathcal P$}
\begin{Lemma}\label{uiP}
Let the conditions of Theorem~\ref{mainTheorem2} hold. Then $\mathcal P$ is $\mathbb{L}^{\infty}(dK)$-bounded, and, therefore,   a uniformly integrable family.
\end{Lemma}
\begin{proof}
{\it Step 1.}  For an arbitrary $q:~[0,T] \to [0,1]$, let us define
\begin{displaymath}
\beta(q) {:=} \sup\limits_{\mathbb Q\in\mathcal M} \mathbb {E^Q}\left[ \int_0^T q(s) |e_s|d\kappa_s\right] = \sup\limits_{\mathbb Q\in\mathcal M, \tau\in\Theta} \mathbb {E^Q}\left[ \int_0^{\tau} q(s) |e_s|d\kappa_s\right].
\end{displaymath}
Note that by Assumption \eqref{asEnd}, 
we have
\begin{equation}\nonumber
\begin{array}{rcl}
\beta(q) &=&\sup\limits_{\mathbb Q\in\mathcal M} \mathbb {E^Q}\left[ \int_0^T q(s) |e_s|d\kappa_s\right]\\
&\leq&\sup\limits_{\mathbb Q\in\mathcal M} \mathbb {E^Q}\left[ \int_0^T q(s) X'_s\varphi d K_s\right]\\
&\leq&\int_0^T q(s)\sup\limits_{\mathbb Q\in\mathcal M} \mathbb {E^Q}\left[  CX'_s\right]d K_s\\
&\leq&CX'_0\int_0^T q(s)d K_s,\\
\end{array}
\end{equation}
where $C\in\mathbb R$ is such that $|\varphi| \leq C$. 
It follows from \cite[Proposition 4.3]{K96} and \cite[Theorem 3.1]{FK} that there exists a nonnegative c\`adl\`ag process $V = \beta(q) + H\cdot S$, such that
\begin{displaymath}
V_t \geq \esssup\limits_{\mathbb Q\in \mathcal M, \tau \in\Theta, \tau\geq t}\mathbb {E^Q}\left[\int_0^{\tau}q(s)|e_s|d\kappa_s | \mathcal F_t \right],\quad t\in[0,T].
\end{displaymath} 
Consequently, $V$ satisfies
$$V_\tau \geq \int_0^\tau q(s)|e_s|d\kappa_s \geq  \int_0^\tau q(s)e_sd\kappa_s,\quad \Pas,\quad\tau \in\Theta,$$
and thus, $(\beta(q), -q)\in\mathcal K$.  As a result, from the definitions of $\mathcal L$ and $\mathcal P$, we get
\begin{equation}\nonumber
\beta(q)\geq \sup\limits_{\rho \in\mathcal P}\int_0^T q(s)\rho(s)dK_s.
\end{equation}
One can see that
 for every $\rho\in\mathcal P$, we have
$\rho \leq f{:=} {CX'_0}$, $d K$-a.e. 

{\it Step 2.} For an arbitrary $q:~[0,T] \to [-1,0]$, let us set 
\begin{displaymath}
\tilde \beta(q) {:=} \sup\limits_{\mathbb Q\in\mathcal M}\mathbb {E^Q}\left[ \int_0^T q(s) (-|e_s|) d\kappa_s \right].
\end{displaymath}
As in {\it Step 1}, we can construct a c\`adl\`ag process $\tilde V = \tilde \beta(q) + H\cdot S$, s.t.
\begin{displaymath}
\tilde V_t \geq \esssup\limits_{\mathbb Q\in\mathcal M, \tau \in\Theta, \tau \geq t}\mathbb {E^Q}\left[ \int_0^T q(s)(-|e_s|) d\kappa_s |\mathcal F_t\right] \geq \esssup\limits_{\mathbb Q\in\mathcal M, \tau \in\Theta, \tau \geq t}\mathbb {E^Q}\left[ \int_0^T q(s)e_s d\kappa_s |\mathcal F_t\right].
\end{displaymath}
This implies that $(\tilde \beta(q), -q)\in\cl\mathcal K$. Therefore, 
\begin{displaymath}
\beta(q) \geq \int_0^T q(s)\rho (s) dK_s,\quad for~every~\rho \in\mathcal P.
\end{displaymath}
Similarly to {\it Step 1}, 
one can see that $\rho \geq -f$, $d K$-a.e.  for every $\rho \in\mathcal P$.

{\it Step 3.}In view of {\it Steps 1 and 2}, uniform integrability of $\mathcal P$ under $dK$ follows from the integrability of $f$ under $d K$.
\end{proof}

\begin{Lemma}\label{uiB}
Let the assumption of Theorem \ref{mainTheorem2} hold and $(x,q)\in\mathring{\mathcal K}$. Then 
$$\left\{r:(y,r)\in\mathcal B(x,q)\right\}$$ is 
 $\mathbb{L}^{\infty}$-bounded, so  a uniformly integrable subset of 
$\mathbb L^1(dK)$.
\end{Lemma}
\begin{proof}
By Lemma \ref{7-17-1}, we deduce the existence of a constant $M>0$, 
such that 
\begin{equation}\nonumber
y\leq M,\quad for~every~(y,r)\in\mathcal B(x,q).
\end{equation}
We conclude that $$\left\{r:(y,r)\in\mathcal B(x,q)\right\}\subseteq \bigcup\limits_{0\leq \lambda \leq M}{\lambda\mathcal P},$$ and thus by Lemma \ref{uiP}, $\left\{r:(y,r)\in\mathcal B(x,q)\right\}$ is a uniformly integrable family.
\end{proof}
\subsection{Closedness of $\mathcal D(x,q)$ for every $(x,q)\in\mathring{\mathcal K}$}
\begin{Lemma}\label{closednessD}
Under the conditions of Theorem~\ref{mainTheorem2}, for every $(x,q)\in \mathring{\mathcal K}$, the set $\mathcal D(x,q)$ is closed in $\mathbb L^0(d\kappa\times\mathbb P)$.
\end{Lemma}
\begin{proof}
Let $(x,q)\in\mathring{\mathcal K}$ 
and $Y$ be an arbitrary element of $\cl \mathcal D(x,q)$.  We claim that there exists $(y,r)\in \mathcal B(x,q)$, such that $Y \in \mathcal Y(y,r)$. Let $Y^n\in\mathcal Y(y^n,r^n)$, $n\geq 1$, be a sequence in $\mathcal D(x,q)$, such that $\lim\limits_{n\to \infty}Y^n = Y$, $(d\kappa\times\mathbb P)$-a.e. 
Since $(y^n,r^n)_{n\geq 1}\subset\mathcal B(x,q)$, which is bounded in the sense of Lemma~\ref{7-17-1}, Komlos' lemma implies the existence of a subsequence of convex combinations
$(\tilde y^n,\tilde r^n)\in \conv((y^n,r^n), (y^{n+1}, r^{n+1}),\dots)$, $n\geq 1$, such that $(\tilde y^n)_{n\geq 1}$ converges to $y$ and $(\tilde r^n)_{n\geq 1}$ converges to $r$, $dK$-a.e. 
Lemma~\ref{uiB} implies that $(\tilde r^n)_{n\geq 1}$ is uniformly integrable. Therefore $(\tilde r^n)_{n\geq 1}$ converges to $r$ in $\mathbb L^1(dK)$. Note that the corresponding sequence of convex combinations of $(Y^n)_{n\geq 1}$, $(\tilde Y^n)_{n\geq 1}$  converges to $Y$, $(d\kappa\times\mathbb P)$-a.e.
Then we have
$$ 
1 \geq \lim\limits_{n\to\infty}\left( x\tilde y^n + \int_0^T q_s\tilde r^n_sdK_s \right) = xy + \int_0^T q_s r_sdK_s.
$$
Therefore, $(y,r)\in \mathcal B(x,q)$.

Let us fix an arbitrary $(x',q')\in\mathcal K$ and $c\in\mathcal A(x',q')$. 
Using Proposition~\ref{prop1}, Fatou's lemma, and Lemma \ref{uiB}, we obtain
\begin{displaymath}
\begin{array}{c}
0\leq \mathbb E\left[\int_0^TY_sc_sds \right] \leq 
\liminf\limits_{n\to\infty} \mathbb E\left[\int_0^T \tilde Y^n_sc_sds \right] \\\leq 
 \lim\limits_{n\to\infty}\left( x'\tilde y^n + \int_0^T q'_s\tilde r^n_sdK_s \right)
 =
  x'y + \int_0^T q'_s r_sdK_s,\\
\end{array}
\end{displaymath} 
where the uniform integrability of $\mathcal B(x,q)$ in needed once again in the last equality.
From Proposition~\ref{prop1}, we conclude that that $Y\in \mathcal Y(y,r)$. 
\end{proof}

\begin{proof}[Proof of Theorem~\ref{mainTheorem2}.] 
Let $(x,q)\in\mathcal K$, $\widehat c(x,q)$ be the minimizer to (\ref{primalProblem}), whose existence and uniqueness are established in Lemma~\ref{existenceUniqueness}. Let us also set
\begin{equation}\label{defHatY}
\widehat Y_{t} {:=} U'({t},\widehat c_{t}(x,q)),\quad (t,\omega)\in[0,T]\times\Omega,\quad and \quad z {:=} \mathbb E\left[\int_0^T\widehat c_s(x,q)\widehat Y_s d\kappa_s \right].
\end{equation} 
Note that the sets $\mathcal A(x,q)$ and ${\cl \mathcal D(x,q)}$ satisfy the conditions of \cite[Theorem~3.2]{MostovyiNec}, which implies that $\widehat Y\in z{\cl \mathcal D(x,q)}$ is the unique solution to the optimization problem
\begin{displaymath}
\inf\limits_{Y\in z{\cl \mathcal D(x,q)}}\mathbb E\left[\int_0^TV(t, Y_s)d\kappa_s\right] = \mathbb E\left[\int_0^TV(t,\widehat Y_s)d\kappa_s\right]\in\mathbb R,
\end{displaymath}
where finiteness follows from Lemma \ref{auxiliaryLemma}.
Note that by Lemma~\ref{closednessD}, $\widehat Y\in\mathcal Y(zy, zr)$ for some $(y,r)\in\mathcal B(x,q)$.
It follows from the definition of $\widehat Y$ in \eqref{defHatY} that 
for $\widehat c(x,q)$ and $\widehat Y$ we have the following relation
\begin{displaymath}
U({t}, \widehat c_{t}(x,q)) = V({t}, \widehat Y_{t}) + \widehat c_{t}(x,q)\widehat Y_{t},
\quad (t,\omega)\in[0,T]\times\Omega,
\end{displaymath}
which together with Lemma~\ref{conjugacy} implies that
\begin{equation}\label{7-17-3}
u(x,q) = v(zy,zr) + z\left(xy + \int_0^Tq_sr_sdK_s\right),
\end{equation}
\begin{displaymath}
\widehat Y(zy,zr)  = \widehat Y 
,\quad (d\kappa\times\mathbb P)-a.e.,
\end{displaymath}
where $\widehat Y(zy,zr)$ is the unique minimizer to the dual problem~(\ref{dualProblem}). 
By (\ref{7-17-3}), $(zy,zr)\in\mathcal E.$ (\ref{7-17-3}) and \cite[Proposition I.5.1]{EkelandTemam} assert that $(zy,zr)\in\partial u(x,q)$. In particular, we get
\begin{displaymath} 
\partial u(x,q) \cap \mathcal E \neq \emptyset,
\end{displaymath}
i.e.,  \eqref{subdifNonempty}. 
Note that even though, e.g., \cite[Corollary 2.2.38 and Corollary 2.2.44]{BarbuPrec} imply that $\partial u(x,q)\neq \emptyset$, 
 over the interior of the effective domain, 
its elements are in $\mathbb R\times (\mathbb L^\infty)^*(dK)$. Relation \eqref{subdifNonempty} shows that $\partial u(x,q)$ contains  at least a bounded element  of $\mathcal E\subseteq\mathcal L\subset \mathbb R\times\mathbb L^1(dK).$

Let $(x,q)\in\mathcal K$ and $(y,r)\in\mathcal L$. Suppose that \eqref{2191}, \eqref{2192}, and \eqref{2193} hold. Then by conjugacy of $U$ and $V$, we get 

\begin{equation}\label{2194}
0 = v(y,r) - u(x,q) + xy + \int_0^T q_sr_sdK_s.
\end{equation}
Lemma \ref{lemBiconjugacy} and 
\cite[Proposition I.5.1]{EkelandTemam} imply that $(y,r)\in\partial u(x,q) \cap\mathcal E$.
Conversely, let $(x,q)\in\mathcal K$ and $(y,r)\in\mathcal L\cap \partial u(x,q).$ Then by \cite[Proposition I.5.1]{EkelandTemam}  and Lemma \ref{lemBiconjugacy}, we deduce that \eqref{2194} holds. Lemma \ref{lemFinu} implies the finiteness of $u(x,q)$, which together with \eqref{2194} results in the finiteness of $v(y,r)$, thus \eqref{2191} holds and $(y,r)\in\mathcal E$. By Lemma \ref{existenceUniqueness}, there exists a unique optimizer $\widehat c(x,q)$, for \eqref{primalProblem}, and $\widehat Y(y,r)$, for \eqref{dualProblem}, respectively. Therefore, from conjugacy of $U$ and $V$, Proposition \ref{prop1}, and \eqref{2194}, we obtain
\begin{displaymath}
\begin{array}{c}
0\leq
 \mathbb E\left[\int_0^T \left(V(t,\widehat Y_t(y,r)) - U(t,\widehat c_t(x,q)) + \widehat Y_t(y,r) \widehat c_t(x,q)\right)d\kappa_t\right] \\
\leq v(y,r) - u(x,q) + xy + \int_0^T q_sr_sdK_s = 0.\\
\end{array}
\end{displaymath}
This implies \eqref{2192} and \eqref{2193}. This completes the proof of the theorem.
\end{proof}

\section{Structure of the dual feasible set}\label{secStructureOfDualDomain}
By Assumption \ref{asACclock}, there exists at most countable subset $(s_k)_{k\in\mathbb N}$ of $[0,T]$, where $\kappa$ has jumps. 
We define $\mathcal{D}'$ the set of non-increasing, left-continuous and adapted processes $D$ that start at $1$ and 
with the property that 
$D_{\theta _0} = 1$,  $ D_T\geq 0$
and that, there exists some $n\in \mathbb N$ such that  $D$ is constant off the discrete grid
$\mathcal T_n{:=} \bigcup\limits_{j = 1}^{2^n}\{s_j\}\cup\left\{\frac{k}{2^n}T, k = 0, \dots,2^n\right\}$.


\begin{Lemma}\label{lemSt2}
 Let $\mathcal G'{:=} \{ZD = (Z_tD_t)_{t\in[0,T]}:~Z\in\mathcal Z', D\in\mathcal D' \}$. Assume    the conditions of Proposition \ref{prop1} hold. Then $\mathcal G'$ is convex.

\end{Lemma}
\begin{proof}
Let $Z^1D^1$ and $Z^2D^2$ are the elements of $\mathcal G'$ and let $\lambda \in(0,1)$
We need to show that $\lambda Z^1D^1 + (1-\lambda)Z^2D^2 = ZD$
for some $Z\in\mathcal Z'$ and $D\in\mathcal D'$. There exists $n\in\mathbb N$, such that $D^1$ and $D^2$ decrease at most on $\mathcal T_n$. Let  $t_k$'s be the elements of $\mathcal T_n$ arranged in an increasing order. 
Let us define $Z_0 = D_0 = 1$ and for every $k \in\{0,\dots, 2^n-1\}$,~with 
$$
A_t {:=} \lambda Z^1_{t_k}D^1_{t} +(1-\lambda)Z^2_{t_k}D^2_{t}
\quad and\quad
\alpha_t {:=} \frac{\lambda Z^1_{t_k}D^1_{t}}{A_t}1_{\{A_t\neq 0\}} + \frac{\lambda Z^1_{t_k}}{\lambda Z^1_{t_k} + (1-\lambda) Z^2_{t_k}}1_{\{A_t= 0\}},$$
(note that $A_t = 0$ if and only if both $D^1_t=0$ and $D^2_t=0$) we  set
\begin{equation}\label{351}
\begin{array}{rclcl}
Z_t &{:=}& Z_{t_k}\left(\alpha_{t_k+}\frac{Z^1_t}{Z^1_{t_k}} + (1 - \alpha_{t_k+})\frac{Z^2_t}{Z^2_{t_k}} \right), &for& t\in(t_k, t_{k+1}],\\
D_t &{:=}& \frac{A_{t_k+}}{Z_{t_k}},& for& t\in(t_k, t_{k+1}].\\
\end{array}
\end{equation}
One can see that $ZD = \lambda Z^1D^1 + (1-\lambda)Z^2D^2$ and that $Z\in\mathcal Z'$, see e.g., \cite{FK}, \cite{Rohlin10}, \cite{Kardaras_emery}, and \cite{ChristaJosef15} for discussions of the sets of processes with similar convexity-type properties to the one given in \eqref{351}. To show that $D\in\mathcal D'$, for $k\geq 1$, 
we observe that 
\begin{displaymath}
\begin{array}{rcl}
D_{t_k+} &=& \frac{A_{t_k+}}{Z_{t_k}} \\
&=& \frac{{\lambda Z^1_{t_k}D^1_{t_k+} +(1-\lambda)Z^2_{t_k}D^2_{t_k+}}}{Z_{t_{k-1}}\left(\alpha_{t_{k-1}+}\frac{Z^1_{t_k}}{Z^1_{t_{k-1}}} + (1 - \alpha_{t_{k-1}+})\frac{Z^2_{t_k}}{Z^2_{t_{k-1}}} \right)} \\
&\leq &\frac{{\lambda Z^1_{t_k}D^1_{t_k} +(1-\lambda)Z^2_{t_k}D^2_{t_k}}}{Z_{t_{k-1}}\left(\alpha_{t_{k-1}+}\frac{Z^1_{t_k}}{Z^1_{t_{k-1}}} + (1 - \alpha_{t_{k-1}+})\frac{Z^2_{t_k}}{Z^2_{t_{k-1}}} \right)}\\
&= &\frac{{\lambda Z^1_{t_k}D^1_{t_k} +(1-\lambda)Z^2_{t_k}D^2_{t_k}}}{Z_{t_{k-1}}\left(\frac{\lambda Z^1_{t_{k-1}}D^1_{t_k}}{A_{t_{k-1}+}}\frac{Z^1_{t_k}}{Z^1_{t_{k-1}}} 
+ \frac{(1-\lambda) Z^2_{t_{k-1}}D^2_{t_k}}{A_{t_{k-1}+}}\frac{Z^2_{t_k}}{Z^2_{t_{k-1}}} \right)}1_{\{A_{t_{k-1}+}> 0\}} +0\cdot1_{\{A_{t_{k-1}+}= 0\}}\\
&= &\frac{A_{t_{k-1}+}}{Z_{t_{k-1}}}\frac{{\lambda Z^1_{t_k}D^1_{t_k} +(1-\lambda)Z^2_{t_k}D^2_{t_k}}}{\lambda D^1_{t_k}Z^1_{t_k}
+ (1-\lambda) D^2_{t_k}Z^2_{t_k}}1_{\{A_{t_{k-1}+}> 0\}}\\
&= &\frac{A_{t_{k-1}+}}{Z_{t_{k-1}}}1_{\{A_{t_{k-1}+}> 0\}}\\
&= &D_{t_k}.\\
\end{array}
\end{displaymath}
Above, the inequality  follows from the monotonicity of $D^1$ and $D^2$.
Therefore, $D$ is nonincreasing. Also, clearly $D$ is nonnegative. Thus  $ZD\in \mathcal G'$. 
\end{proof}
The following lemma is an extension of Lemma \ref{9-5-1}  and amounts to a first layer of convexification of the  set $\Upsilon$, i.e., of the budget constraints.

\begin{Lemma}\label{lemSt1}Let   the conditions of Proposition \ref{prop1} hold,  $(x,q)\in\mathcal K$, and $c$ is a nonnegative optional process. Then $c\in\mathcal A(x,q)$ if and only if
\begin{equation}\label{331}
\mathbb E\left[\int_0^Tc_sD_sZ_sd\kappa_s \right]\leq
x + \mathbb E\left[\int_0^Tq_se_sD_sZ_sd\kappa_s \right],\quad
for~every~Z\in\mathcal Z'~and~D \in \mathcal D'.
\end{equation}
\end{Lemma}
\begin{proof}
The idea is to use the assertion of Lemma \ref{9-5-1} and to approximate a given $D\in\mathcal D'$ by (finite) linear combinations of the elements of $\Upsilon$, where $\Upsilon$ is defined in \eqref{defUpsilon}.

%
%
%
For a stopping time $\tau$, 
let us denote $$\Lambda^{\tau} {:=} 1_{[0,\tau]}\in\Upsilon,$$
and fix $D\in\mathcal D'$. Then there exists $l\in\mathbb N$, such that $D$ has has jumps at most 
on $\{t_0,t_1,\dots, 
t_l\}$ for some increasing $t_i$'s. 
For every $j \in\{ 0, \dots, l
\}$, $k \in\{0,\dots 2^n\}$, and $n\in\mathbb N$, let us set
\begin{equation}\nonumber
\begin{array}{rcl}
A_{k, n, j} &{:=}& \left\{ \omega:  ~D_{t_j}(\omega) >0~and~\frac{D_{t_j+}(\omega)}{D_{t_j}(\omega)} \in \left(\frac{k-1}{2^n}, \frac{k}{2^n}\right]\right\},\\
\tau^{k, n, j}&{:=}&T1_{A_{k,n,j}} + t_j1_{A^c_{k,n,j}},\\
\end{array}
\end{equation}
Note that $D_0 = 1$ by definition of $\mathcal D'$, $A_{k, n, j}\in\mathcal F_{t_j}$, and 
$$\frac{k-1}{2^n}\Lambda^{\tau^{k,n,j}}_{t_j+} = \left\{
\begin{array}{lcl}
\frac{k-1}{2^n}&on&A_{k,n,j}\\
0&on& A^c_{k,n,j}\\
\end{array}
 \right..$$
 By construction, we have $$D_{t_j+} = D_{t_j}\lim\limits_{n\to\infty}\sum\limits_{k = 1}^{2^n}\frac{k-1}{2^n}\Lambda^{\tau^{k,n,j}}_{t_j+},$$ where the sequence 
 $$K^n_{t_j+}{:=} \sum\limits_{k = 1}^{2^n}\frac{k-1}{2^n}\Lambda^{k,n,j}_{t_j+},\quad n\in\mathbb N,$$
 is increasing on $\{D_{t_j}>0\}$, i.e.,
 $$K^n_{t_j+}1_{\{D_{t_j}>0\}}\uparrow \frac{D_{t_j+}}{D_{t_j}}1_{\{D_{t_j}>0\}}.$$
 Thus, for an arbitrary $j \in\{0, \dots, l\}$, we have constructed a sequence of elements of $\Upsilon$, whose finite linear combinations monotonically increase to $\frac{D_{t_j+}}{D_{t_j}}$ on $\{D_{t_j}>0\}$ (i.e., if $j<l$, we have approximated $D$ on the interval $(t_{j}, t_{j+1}]$).

 In order to construct a sequence that approximates $D$ at every point of its potential jumps, we first observe that for two stopping times $\tau$ and $\sigma$, we have
 $$\Lambda^\tau\Lambda^\sigma = 1_{[0, \tau]}1_{[0,\sigma]} = 1_{[0,\tau\wedge\sigma]} = \Lambda^{\tau\wedge \sigma}.$$
Therefore, for every $n\in\mathbb N$,
\begin{equation}\label{361}
K^n_{t_j+}K^n_{t_{j+1}+} =\left(\sum\limits_{k = 1}^{2^n}\frac{k-1}{2^n}\Lambda^{k,n,j}_{t_j+}\right)\left(\sum\limits_{k = 1}^{2^n}\frac{k-1}{2^n}\Lambda^{k,n,j+1}_{t_{j+1}+}\right) =\sum\limits_{i=1}^{4^n}\lambda^{n,i} \Lambda^{\sigma_{n,i}}_{t_{j+1}+},
\end{equation}
for some  finite sequences of stopping times $(\sigma_{n,i})_{i=1}^{4^n}$ and $[0,1)$-valued constants $(\lambda^{n,i})_{i=1}^{4^n}$. 
Here $\Lambda^{\sigma_{n,i}}$ are such that for {\it both} $t = t_j$ and $t=t_{j+1}$ on $\{D_{t_{j+1}}>0\}$, we have
$$\lim\limits_{n\to\infty}\sum\limits_{i=1}^{4^n}\lambda^{n,i} \Lambda^{\sigma_{n,i}}_{t+} = \frac{D_{t+}}{D_t}.$$
Similarly,
%
 with $r(t){:=} \max\{i:~t_i< t\}$, let us define
$$D^n_t {:=} \prod\limits_{j = 0}^{r(t)}K^n_{t_j+}1_{\{D_{t_j}>0\}},\quad t\in[0,T], n\in\mathbb N.$$ 
As in \eqref{361}, for every $n\in\mathbb N$, $D^n$ can be written as a finite linear combination of $\Lambda$'s, such that $D^n_{t+}\uparrow D_{t+}$ for every $t\in\{t_0, \dots, t_{l}\}.$

Finally, \eqref{331} can be obtained from Lemma \ref{9-5-1} by the approximation of $D$ by $D^n$'s as above and via  the monotone convergence theorem (applied separately to $(e_t)^{+}$ and $(e_t)^{-}$). 
\end{proof}
\begin{Corollary}\label{corRzd}
Let   the conditions of Proposition \ref{prop1} hold. Then, for every pair $Z\in\mathcal Z'$ and $D\in\mathcal D'$, 
there exists $r^{ZD}\in\mathbb L^1(dK)$, such that $ZD\in\mathcal Y(1,r^{ZD})$, where $(1,r^{ZD})\in\mathcal L$. 
\end{Corollary}
\begin{proof}
The existence of $r^{ZD}$, such that $ZD\in\mathcal Y(1, r^{ZD})$ follows from Lemma \ref{lemSt1} (equation \eqref{331}) and the approximation procedure in Lemma \ref{lemSt1} (again, applied separately to $(e_t)^{+}$ and $(e_t)^{-}$) combined with the monotone convergence theorem.
As, the left-hand side in \eqref{331} is nonnegative, 
$(1,r^{ZD})\in\mathcal L$. 

\end{proof}

For a given $Z\in\mathcal Z'$ and $D\in\mathcal D'$, let us recall that $r^{ZD}$ is given in Corollary \ref{corRzd}. 
We set $$\mathcal {B'}(x,q) {:=} \left\{ (y, yr^{ZD})\in\mathcal B(x,q):~y> 0, Z\in\mathcal Z', D\in\mathcal D'\right\},$$

$$\mathcal G(x,q) {:=} \left\{yZ'D'\in\mathcal D(x,q):~(y, yr^{ZD})\in \mathcal {B'}(x,q)\right\},\quad (x,q)\in\mathring{\mathcal K},$$ 


\begin{Lemma}\label{lemSt3}
Let   the conditions of Proposition \ref{prop1} hold, then for every $(x,q) \in\mathring{\mathcal K}$, the closure of the convex, solid hull of $\mathcal G(x,q)$ in $\mathbb L^0$ coincides with $\cl\mathcal D(x,q)$. 
\end{Lemma}
\begin{proof}
Let $(x,q)\in\mathring{\mathcal K}$ be fixed. 
Along the lines of the proof of Lemma \ref{5181}, one can show that
\begin{equation}\nonumber
(\mathcal G(x,q))^o = \mathcal A(x,q).
\end{equation}
Therefore, 
$$(\mathcal G(x,q))^{oo} = (\mathcal A(x,q))^o = \cl\mathcal D(x,q),$$
where in the last equality we have used the conclusion of Lemma \ref{5181}. 
 As $\mathcal G(x,q)\subset \cl\mathcal D(x,q)$, the assertion of the lemma follows from  the bipolar theorem of Brannath and Schchermayer, \cite[Theorem 1.3]{BranSchach}.
\end{proof}
%

\begin{Corollary}\label{lemSt4}
Let the conditions of Proposition \ref{prop1} hold and $(x,q)\in\mathring{\mathcal K}$. Then for  
every \underline{maximal} element $Y$ of $\cl\mathcal D(x,q)$ there exists $y^n\geq 0$, $Z^n\in\mathcal Z'$, $D^n\in\mathcal D'$, $n\in\mathbb N$, such that $(y^nZ^nD^n)_{n\in\mathbb N}\subset\mathcal G(x,q)$ and 
\begin{equation}\nonumber
\begin{array}{rcll}
Y 
&=& \lim\limits_{n\to\infty} y^nZ^nD^n,& \quad (d\kappa\times \mathbb P)-a.e.~and~on~ \bigcup\limits_{n\in\mathbb N}\mathcal T_n.
\end{array}
\end{equation}
\end{Corollary}
\begin{proof}
The $(d\kappa\times \mathbb P)$-a.e. convergence follows from Lemma \ref{lemSt3}.  By passing to subsequences of convex combinations, we also deduce the convergence on $\bigcup\limits_{n\in\mathbb N}\mathcal T_n$.
\end{proof}
\begin{Remark}
It follows from Corollary \ref{lemSt4} and Fatou's lemma that the maximal elements of $\cl\mathcal D(x,q)$ are strong supermartingales. Moreover, every maximal element of $\cl\mathcal D(x,q)$ is an optional strong supermartingale deflator, which is optional strong supermartingale $Y$, such that $XY$ is an optional strong supermartingale for every $X\in\mathcal X(1)$. We refer to \cite[Appendix 1]{DelMey82} for a general characterization and to \cite{CzichSchachOSSup} for results on strong optional supermartingales as limits of martingales. The following section gives a more refined characterization of the dual minimizer.
\end{Remark}

\section{Complementary slackness}\label{secSlackness}
For better readability of this section, we recall some notations and results that will be used below. 
Throughout this section, $(x,q)\in \mathring{\mathcal{K}}$ will be fixed, $\widehat c = \widehat c(x,q)$ is the optimizer to \eqref{primalProblem}, $\widehat V$ is the corresponding wealth process, i.e., 
 \begin{equation}\label{5211}
\widehat V=x+\int_0^{\cdot}\widehat H_sdS_s-\int_0^{\cdot}\widehat c_sd\kappa_s+\int_0^{\cdot} q_se_sd\kappa_s,
\end{equation}
where $\widehat H$ is some $S$-integrable process, 
$\widehat Y$ be such that $\widehat Y_t = U'(t,\widehat c_t)$, $(d\kappa\times\mathbb P)$-a.e., i.e., $\widehat Y$ is the optimizer to \eqref{dualProblem} for some 
$(y,r)\in\mathcal E\cap \partial u(x,q)$, $\widehat Y\in\mathcal Y(y, r)$ and 
$$\mathbb E\left[\int_0^TV(t,\widehat Y_t)d\kappa_t\right] = \inf\limits_{Y\in \mathcal Y(y,r)}\mathbb E\left[\int_0^TV(t, Y_t)d\kappa_t\right].$$
By Corollary \ref{lemSt4}, $\widehat Y$  can be approximated by a sequence $y^nD^nZ^n$, $n\in\mathbb N$, where $y^n$ is a nonnegative constant, $D^n\in\mathcal D'$, and $Z^n\in\mathcal Z'$, $n\in\mathbb N$.
In what follows, for any right-continuous  increasing process $A$ satisfying $$A_t=0,\quad 0-\leq t\leq \theta _{0}-,\quad A_T=1,$$ 
i.e., for any probability measure $dA$ on  the  closed interval $[\theta _0, T]$ (extended to $[0,T]$) we will associate a process $D$ which is left-continuous  and decreasing
$$D_t{:=} 1-A_{t-}=dA([t,T]), \ \ 0\leq t\leq T.$$
One can also think that $ D_{T+}=0,$ although  this is not necessary.  It is clear that  such $A\leftrightarrow D$ are in  bijective  correspondence.  Below, all processes $A$'s and $D$'s (with indexes) will be in such bijective correspondence, except for the case of  the limiting  process $\widehat A$ (which is right-continuous) and  the limiting process $\widehat D$ (that may be not  left-continuous). The  will be in a similar but more subtle correspondence. More precisely:
\begin{Theorem}\label{ExactSlackness}
Let the conditions of Theorem \ref{mainTheorem2}  hold. Let  $\widehat V$ be the optimal wealth process, $\widehat c$  the optimal consumption, and $\widehat Y$ be the dual minimizer satisfying the assertions of Theorem \ref{mainTheorem2}. Then, there exists a strong supermartingale $\hat Z=\lim\limits_{n\to\infty} Z^n$ (for $Z_n\in \mathcal{D}'$, where the limit is in the sense of \cite{CzichSchachOSSup}) and a right-continuous  increasing process $\widehat A$ with
$$ \widehat A_t=0\ \  \forall 0-\leq t \leq \theta _0-,\ \ \  \widehat A_T= 1, $$
and a decreasing process $\widehat D$ with $\widehat D_t=1,\ \ 0\leq t\leq {\theta _0}$ and  satisfying the complementary slackness condition 
\begin{equation}\label{comp-slack}
 \mathbb{P}\left (\widehat D_t\in \left[1-\widehat A_t, 1-\widehat A_{t-}\right], \  \forall \ \theta _0 \leq t\leq T)\right)=1,\ \ 
 \mathbb{P}\left (\int _ {[\theta _0,T)}1_{\{\widehat V_{t-}\not= 0, \widehat V_t \not =0\}} d\widehat A_t\right)=0,
 \end{equation}
such that the dual minimizer can be decomposed as
$$\hat Y= y \hat Z\hat D.$$
\end{Theorem}

The proof of the Theorem \ref{ExactSlackness} is split in several results. 
\begin{Lemma}\label{lem5211} Let the conditions of Theorem \ref{mainTheorem2}  hold. With  
$$z{:=} xy+\int_0^T q_sr_sdK_s,$$ there exist $y^n>0$,  $Z^n\in\mathcal Z'$, and $D^n\in\mathcal D'$, such that $y^nZ^nD^n\in z\mathcal G(x,q)$, $n\in\mathbb N$, and
\begin{equation}\label{5231}
\begin{array}{c}
\widehat Y 
= \lim\limits_{n\to\infty} y^nZ^nD^n,\quad (d\kappa\times \mathbb P)-a.e.\\
y^n\to y>0;\\
\end{array}
\end{equation}
For $A^n$, $n\in\mathbb N$, being  in relation to $D^n$ exactly as described before Theorem \ref{ExactSlackness} we have
\begin{equation}\label{86*}\mathbb{E}\left [\int _{\theta _0}^T \widehat V_t Z^n_t dA^n _t \right]=\mathbb{E}\left [Z^n_T\int _{\theta _0}^T \widehat V_t dA^n _t \right]\rightarrow 0.
\end{equation}
\end{Lemma}
\begin{proof} Optimality of $\widehat Y$ and Corollary \ref{lemSt4} imply \eqref{5231}.  
By \eqref{2192} and Fatou's lemma, we get
\begin{equation}\label{eqndefz}
z = \mathbb E\left[\int_0^T \widehat Y_s\widehat c_sd\kappa_s \right] \leq \liminf\limits_{n\to\infty}y^n\mathbb E\left[ \int_0^T Z^n_sD^n_s\widehat  c_sd\kappa_s\right].
\end{equation}
Let us fix $n\in\mathbb N$ and consider $\mathbb E\left[ \int_0^T Z^n_sD^n_s\widehat  c_sd\kappa_s\right]$. Using localization and integration by parts (along the lines of the proof of Lemma \ref{6-16-1}), we have
\begin{displaymath}\begin{array}{rcl}
\mathbb E\left[ \int_0^T Z^n_sD^n_s\widehat  c_sd\kappa_s\right] &=& \mathbb E\left[Z^n_T \int_0^T D^n_{s}\widehat  c_sd\kappa_s\right] \\
&=& \mathbb E\left[Z^n_T \int_0^T \left(1 - A^n_{s-}\right)\widehat  c_sd\kappa_s\right] \\
&=& \mathbb E\left[Z^n_T \left( \int_0^T\widehat  c_sd\kappa_s - \int_0^T A^n_{s-} \widehat c_sd\kappa_s\right)\right] \\
&=& \mathbb E\left[Z^n_T \left( \int_0^T\widehat  c_sd\kappa_s - A^n_{T} \int_0^T \widehat c_sd\kappa_s + \int_0^T(\int_0^t\widehat c_sd\kappa_s)dA^n_t \right)\right] \\
&=& \mathbb E\left[Z^n_T \int_0^T(\int_0^t\widehat c_sd\kappa_s)dA^n_t \right].\\
\end{array}
\end{displaymath}
The latter expression, using \eqref{5211} and with   $\bar X {:=} ||q||_{\mathbb L^\infty(dK)}X''$ (where in turn $X''$ is given by the assertion $(v)$ of Lemma \ref{lemma1}), we can rewrite as
\begin{equation}\label{5212}
\mathbb E\left[Z^n_T \int_0^T\left((x + \int_0^{t}\widehat H_sdS_s +\bar X_t) + \left(\int_0^tq_se_sd\kappa_s - \bar X_t\right) - \widehat V_t\right)dA^n_t \right].
\end{equation}
Let us denote
\begin{equation}\nonumber
\begin{array}
{rcl}
T_1 &{:=} & \mathbb E\left[Z^n_T \int_0^T\left(x + \int_0^{t}\widehat H_sdS_s +\bar X_t\right)dA^n_t \right] ,\\
T_2 &{:=} & \mathbb E\left[Z^n_T \int_0^T\left(\int_0^tq_se_sd\kappa_s - \bar X_t\right)dA^n_t \right].\\
\end{array}
\end{equation}
It follows from nonnegativity of $\widehat V$ in \eqref{5211}, Lemma \ref{lemma1}, and nonnegativity of $\widehat c$ that 
\begin{equation}\label{5232}
x + \int_0^{t}\widehat H_sdS_s + \bar X_t \geq \int_0^t\widehat c_sd\kappa_s + \bar X_t -\int_0^t q_se_sd\kappa_s \geq \bar X_t -\int_0^t q_se_sd\kappa_s \geq0,\quad t\in[0,T],
\end{equation}
i.e., the integrand in $T_1$ is nonnegative. Let ${\mathbb Q^n}$ be the probability measure, whose density process with respect to $\mathbb P$ is $Z^n$. As $(x + \int_0^{\cdot}\widehat H_sdS_s + \bar X)$ and $\bar X$ are local martingales under ${\mathbb Q^n}$, by \cite[Theorem III.27, p. 128]{Pr}, $A^n_{-}\cdot(x + \int_0^{\cdot}\widehat H_sdS_s + \bar X)$ and $A^n_{-}\cdot\bar X$ are local martingales. Let $(\tau_k)_{k\in\mathbb N}$ be a localizing sequence for both $A^n_{-}\cdot(x + \int_0^{\cdot}\widehat H_sdS_s + \bar X)$ and $A^n_{-}\cdot\bar X$. 
By the monotone convergence theorem and integration by parts, we get
\begin{equation}\label{5214}
\begin{array}{rcl}
T_1 &=& \lim\limits_{k\to\infty}\mathbb E^{\mathbb Q^n}\left[\int_0^{\tau_k}\left(x + \int_0^{t}\widehat H_sdS_s +\bar X_t\right)dA^n_t \right] \\
&=& \lim\limits_{k\to\infty}\left(\mathbb E^{\mathbb Q^n}\left[ \int_0^{\tau_k}(-A^n_{t-})d(x + \int_0^{t}\widehat H_sdS_s + \bar X_t) + A^n_{\tau_k} \left( x + \int_0^{\tau_k}\widehat H_sdS_s +\bar X_{\tau_k}\right)\right]\right) \\
&=&\lim\limits_{k\to\infty}\mathbb E^{\mathbb Q^n}\left[A^n_{\tau_k} \left( x + \int_0^{\tau_k}\widehat H_sdS_s +\bar X_{\tau_k}\right)\right].\\
\end{array}
\end{equation}
Let us consider $T_2$. With $\mathcal E^q {:=} \int_0^{\cdot} q_se_sd\kappa_s$, Lemma \ref{lemma1} implies positivity of $\mathcal E^q_t - \bar X_t$, $t\in[0,T]$, which in turn allows to invoke the monotone convergence theorem, and we obtain 
\begin{equation}\nonumber
\begin{array}{rcl}
T_2 &=& \mathbb E^{\mathbb Q^n}\left[ \int_0^T\left(\mathcal E^q_t - \bar X_t\right)dA^n_t \right] \\
&=&\lim\limits_{k\to\infty}\mathbb E^{\mathbb Q^n}\left[ \int_0^{\tau_k}\mathcal E^q_t dA^n_t - \int_0^{\tau_k}\bar X_tdA^n_t \right].\\
\end{array}
\end{equation} 
By positivity of $\bar X$ and the monotone convergence theorem, we have $$\lim\limits_{k\to\infty}\mathbb E^{\mathbb Q^n}\left[\int_0^{\tau_k}\bar X_tdA^n_t \right]=\mathbb E^{\mathbb Q^n}\left[\int_0^{T}\bar X_tdA^n_t \right].$$
Therefore, $\lim\limits_{k\to\infty}\mathbb E^{\mathbb Q^n}\left[ \int_0^{\tau_k}\mathcal E^q_t dA^n_t\right]=\mathbb E^{\mathbb Q^n}\left[ \int_0^{T}\mathcal E^q_t dA^n_t\right]$. We deduce that 
\begin{displaymath}
\mathbb E^{\mathbb Q^n}\left[\int_0^{\tau_k}\bar X_tdA^n_t \right] = \mathbb E^{\mathbb Q^n}\left[-(A^n_{-}\cdot\bar X)_{\tau_k} +\bar X_{\tau_k}A^n_{\tau_k}\right] = \mathbb E^{\mathbb Q^n}\left[\bar X_{\tau_k}A^n_{\tau_k}\right].
\end{displaymath}
Whereas, in the other term in $T_2$, we get
\begin{displaymath}
\begin{array}{rcl}
\mathbb E^{\mathbb Q^n}\left[ \int_0^{T}\mathcal E^q_t dA^n_t\right] &= &\mathbb E^{\mathbb Q^n}\left[\mathcal E^q_{T}A^n_{T} - (A^n_{-}\cdot \mathcal E^q)_{T}\right] \\
&= &\mathbb E^{\mathbb Q^n}\left[\mathcal E^q_{T} - (A^n_{-}\cdot \mathcal E^q)_{T}\right] \\
&= &\mathbb E^{\mathbb Q^n}\left[((1- A^n_{-})\cdot \mathcal E^q)_{T}\right] \\
&= &\mathbb E^{\mathbb Q^n}\left[(D^n\cdot \mathcal E^q)_{T}\right].\\
\end{array}
\end{displaymath}
using integration by parts and localization, as in the proof of Lemma \ref{6-16-1}, we can rewrite the latter expression as
$$
\mathbb E\left[\int_0^T D^n_sZ^n_sq_se_sd\kappa_s\right].
$$
We conclude that
$$T_2 = -\lim\limits_{k\to\infty}\mathbb E^{\mathbb Q^n}\left[\bar X_{\tau_k}A^n_{\tau_k}\right] +\mathbb E\left[\int_0^T D^n_sZ^n_sq_se_sd\kappa_s\right].$$
Combining this with \eqref{5214}, we obtain 
$$T_1 + T_2  = \mathbb E\left[\int_0^T D^n_sZ^n_sq_se_sd\kappa_s\right] + \lim\limits_{k\to\infty}\mathbb E^{\mathbb Q^n}\left[A^n_{\tau_k}\left(x + \int_0^{\tau_k}\widehat H_sdS_s +\bar X_{\tau_k} \right)\right] - \lim\limits_{k\to\infty}\mathbb E^{\mathbb Q^n}\left[A^n_{\tau_k}\bar X_{\tau_k} \right].$$
As both limits in the right-hand side exist and by positivity of $\left(x + \int_0^{\tau_k}\widehat H_sdS_s +\bar X_{\tau_k} \right)$, established in \eqref{5232}, we can bound the difference of the limits as
\begin{displaymath}\begin{array}{c}
 \lim\limits_{k\to\infty}\mathbb E^{\mathbb Q^n}\left[A^n_{\tau_k}\left(x + \int_0^{\tau_k}\widehat H_sdS_s +\bar X_{\tau_k} \right)\right] - \lim\limits_{k\to\infty}\mathbb E^{\mathbb Q^n}\left[A^n_{\tau_k}\bar X_{\tau_k} \right] \\
 \leq 
  \lim\limits_{k\to\infty}\mathbb E^{\mathbb Q^n}\left[x + \int_0^{\tau_k}\widehat H_sdS_s +\bar X_{\tau_k}\right] - \lim\limits_{k\to\infty}\mathbb E^{\mathbb Q^n}\left[A^n_{\tau_k}\bar X_{\tau_k} \right] \\
   \leq 
   \lim\limits_{k\to\infty}\mathbb E^{\mathbb Q^n}\left[x + \int_0^{\tau_k}\widehat H_sdS_s\right] + 
 \lim\limits_{k\to\infty}\mathbb E^{\mathbb Q^n}\left[(1-A^n_{\tau_k})\bar X_{\tau_k} \right].\\
\end{array}
\end{displaymath}
By definition of $\mathcal M'$, $\bar X$ is a uniformly integrable martingale under $\mathbb Q^n$. Therefore, as $(1-A^n_{\tau_k})$ is bounded, we can pass the limit inside of the expectation to obtain $\lim\limits_{k\to\infty}\mathbb E^{\mathbb Q^n}\left[(1-A^n_{\tau_k})\bar X_{\tau_k} \right] = 0$. In turn $x + \int_0^{\cdot}\widehat H_sdS_s$ is a supermartingale under $\mathbb Q^n$ (see the argument in the proof of Lemma \ref{6-16-1}). We conclude that 
$$
T_1 + T_2 \leq x + \mathbb E\left[\int_0^T D^n_sZ^n_sq_se_sd\kappa_s\right] = x + \int_0^T q_s\rho^n_sdK_s,
$$
for some $\rho^n$, which is well-defined by Corollary \ref{corRzd}, and such that $y^n(1,\rho^n)\in z\mathcal B(x,q)$, since $y^nZ^nD^n\in z\mathcal G(x,q)$.   
Thus, from \eqref{eqndefz} and \eqref{5212}, we get
\begin{displaymath}
\begin{array}{rcl}
z&\leq &(x + \int_0^T q_s\rho^n_sdK_s)y^n  -\lim\limits_{n\to\infty}y^n\mathbb E\left[Z^n_T \int_0^T\widehat V_tdA^n_t \right].\\
\end{array}
\end{displaymath}
By optimality of $\widehat Y$, $z\geq (x + \int_0^T q_s\rho^n_sdK_s)y^n\geq 0$. Therefore, by nonnegativity of $\mathbb E\left[Z^n_T \int_0^T\widehat V_tdA^n_t \right]$, and since $y^n$ converges to a strictly positive limit, we conclude that
$$\lim\limits_{n\to\infty}\mathbb E\left[Z^n_T \int_0^T\widehat V_tdA^n_t \right] = 0.$$
Applying integration by parts and localization we deduce the assertion of the lemma.
\end{proof}
\begin{Remark} We emphasize again that  $dA^n$ are probability measures on $[\theta _0,T]$, which can have mass at the endpoints, and 
 $D^n_t=1-A^n_{t-}, \theta _0\leq t\leq T$, $D^n_{T+}=0$. 
\end{Remark}

In the subsequent part, we will follow the notations of Lemma \ref{lem5211} and we will work with a further subsequence, still denoted by $n$, such that
$y_n\rightarrow  y>0$, 
\begin{equation}
\label{cone-product}
Z^nD^n\rightarrow \frac{\widehat Y}{y}, \quad (d\kappa\times\mathbb P)-a.e, 
\end{equation} and
\begin{equation}\label{fast}
\sum _{k=n}^{\infty}\mathbb E\left[Z^n_T \int_0^T\widehat V_tdA^n_t \right]\leq \frac{2^{-n}}{n},
\end{equation}
where the existence of a subsequence satisfying \eqref{fast} follows from \eqref{86*}.
The next results is a Komlos-type lemma, largely based on the results in  \cite{CzichSchachOSSup}, applied to the \emph{double-sequence} of processes $(Z^n, \left (D^{n}\right)^{-1})$.   We observe that the process $\left ( D^n\right )^{-1}$ takes values in $[1, \infty]$ is increasing,  left-continuous,  and satisfies
$$\left ( D^n _t\right )^{-1}=1,\quad  0\leq t\leq \theta _0.$$
\begin{Lemma}\label{5219} 
Let the conditions of Theorem \ref{mainTheorem2}  hold. In the notations of Lemma \ref{lem5211},
for each $n$, there exist  a finite index $N(n)$ and convex weights
$$\alpha _{n,k}\geq 0,\quad   k=n,..., N(n),\quad \sum _{k=n}^{N(n)}\alpha _{n,k}=1,$$ 
and there exists a strong optional super-martingale $\widehat Z$ and a non-decreasing  (not necessarily left-continuous) process $\widehat D$ 
with $\widehat D_t=0, \ \forall 0\leq t\leq \theta_0$, 
such that, simultaneously, 
\begin{enumerate}
\item $\tilde Z^n{:=} \sum _{k=n}^{N(n)}\alpha _{n,k}Z^k\rightarrow \widehat Z$ in the sense of \cite{CzichSchachOSSup} i.e. for any stopping time $0\leq \tau\leq T$ we have 
$$\tilde Z^n _{\tau} {\longrightarrow}^{\mathbb{P}}  \widehat Z_{\tau},$$
and 
\item 
$$ \mathbb{P}\left (\sum _{k=n}^{N(n)}\alpha _{n,k}(D^k_t)^{-1}\rightarrow (\widehat D_t )^{-1}, \quad  \forall  0\leq t \leq T+ \right)=1$$

 We have set $D_{T+}=D^k_{T+}=0$. \end{enumerate}
Furthermore,  with the notation
\begin{equation}\label{a-hat}
 \widehat A _t {=} 1-\widehat D_{t+}, \quad  0\leq t\leq T,\quad \widehat A_{0-}=0, 
\end{equation}
we have the probability measure  $d\widehat A$ on $[\theta _0,T]$ such that
$$\mathbb{P}\left (\widehat D_t\in \left[1-\widehat A_t, 1-\widehat A_{t-}\right], \  \forall \ \theta _0 \leq t\leq T)\right)=1.$$
Denoting by 
\begin{equation}\label{a-tilde-n}
\tilde A_t^n{:=} 1- \left (\sum _{k=n}^{N(n)}\alpha _{n,k}(1-A^k_t)^{-1}\right)^{-1}=1-
\left (\sum _{k=n}^{N(n)}\alpha _{n,k}(D^k_{t+})^{-1}\right)^{-1}, \quad  0\leq t\leq T, \quad \tilde A ^n_{0-}=0,
\end{equation}
the point-wise convergence in time (at all times where there is continuity)  additionally implies 
$$d\tilde A^n\rightharpoonup d\widehat A,\quad \mathbb{P}-a.e.$$ in the sense of weak convergence of probability measures on $[\theta _0, T]$.  \end{Lemma}
\begin{proof} 
The proof reduces to applying  the  Komlos-type results in \cite{CzichSchachOSSup} to the sequence $Z^n$ and  \cite[Proposition 3.4]{campi-schachermayer} to the sequence $(D^{n})^{-1}$, simultaneously. We observe that: 
\begin{enumerate}
\item first, no bounds are needed    for $D^{-1}$'s  since we can apply the unbounded Komlos lemma in \cite[Lemma A1.1]{DS} (and Remark 1 following it) to the proof from \cite[Propositions 3.4]{campi-schachermayer}, and this works even for infinite values (according to  Remark 1 after \cite[Lemma A1.1]{DS}). Also,  predictability  can be replaced by optionality, without any change to the proof.
\item the Komlos arguments can be applied to both sequences $Z^n$ and $D^n$ simultaneously, with the same convex weights. In order to do this,   we first apply Komlos  to one sequence, then replace both original sequences by their  convex combinations with the weights just  obtained, and then apply Komlos again for the other sequence, and update the convex weight to both sequences.
\end{enumerate}

\end{proof}

\begin{Corollary} Let the conditions of Theorem \ref{mainTheorem2}  hold. In the notations of Lemma \ref{5219}, we have the representation
$$\widehat Y= y\widehat Z\widehat D,\quad (d\kappa\times\mathbb P)-a.e.$$
\end{Corollary}
\begin{proof} Consider an observation that, for non-negative 
numbers  $a_k, b_k$, $k=n,\dots, N(n)$, we have
$$\min _{k=n,...,N(n)} \left (\frac{a_k}{b_k}\right)\leq \frac{\sum _{k=n}^{N(n)}\alpha _{n,k}a^k}{\sum _{k=n}^{N(n)}\alpha _{n,k}b_k}\leq  \max _{k=n,...,N(n)} \left (\frac{a_k}{b_k}\right).$$\
We apply this to $a_k=Z^k, b_k=(D^k)^{-1},$ to obtain
$$\min _{k=n,...,N(n)} \left (Z^kD^k\right)\leq 
 \frac{\sum _{k=n}^{N(n)}\alpha _{n,k}Z^k}{\sum _{k=n}^{N(n)}\alpha _{n,k}(D^k)^{-1}}=\tilde Z^n \tilde D^n \leq
  \max _{k=n,...,N(n)} \left (Z^kD^k\right),$$
pointwise a.e.  in the product space, where $\tilde D^n {:=} \frac{1}{\sum _{k=n}^{N(n)}\alpha _{n,k}(D^k)^{-1}}$. 
  Since $Z^nD^n\rightarrow \widehat Y/  y$, $(d\kappa\times \mathbb P)$-a.e., we conclude that both
  $\left(\min _{k=n,...,N(n)} \left (Z^kD^k\right)\right)_{n\in\mathbb N}$ and  $\left(\max _{k=n,...,N(n)} \left (Z^kD^k\right)\right)_{n\in\mathbb N}$ converge to $\frac{\widehat Y}{y}$, $(d\kappa\times \mathbb P)$-a.e.,
  therefore
    \begin{equation}
    \label{conv-product}\tilde Z^n \tilde D^n    \rightarrow \frac{\widehat Y}{y},\quad (d\kappa\times \mathbb P)-a.e.
   \end{equation}
Using  \eqref{cone-product} above, if we did not plan to identify the limit $\widehat Z$   as a strong-supermartingale, but only as a $(d\kappa\times \mathbb P)$-a.e. limit in the product space,  we could only  apply Komlos arguments to the single  sequence $D^{-1}$,  to conclude convergence of the other  convex combination $\tilde Z^n$ to $\frac{\widehat Y}{y\widehat D}$, defined up to a.e. equality in the product space.

  With our  (stronger) double Komlos argument, we have that, in addition to \eqref{cone-product} we have   \begin{equation}
  \label{conv-tau}\widehat Z_{\tau}\widehat D_{\tau}=\mathbb{P}-\lim _n \tilde Z^n_{\tau}\tilde D^n _{\tau},\quad \textrm{for\ every~[0,T]-valued~stopping \ time \ }\tau,
  \end{equation}
  (where $\widehat Z$ is a strong supermartingale, and $\widehat D$ is well defined at all times).

  It remains to prove that $\widehat Y/y=\widehat Z\widehat D$, $(d\kappa\times\mathbb P)$-a.e. We point out that the convergence \eqref{conv-tau} is topological. Let us consider an arbitrary optional set $O\subset \Omega \times [0,T]$ and fix an upper bound $M$.  It follows from \eqref{conv-product} that
  $$1_O \min \left \{\tilde Z^n \tilde D^n, M \right \}\rightarrow  1_O \min \left \{\frac{\widehat Y}{y}, M \right \},\quad (d\kappa\times \mathbb P)-a.e.
$$
Therefore, we get
\begin{equation}\label{hatY}\mathbb{E}\left [\int _0^T  1_O(t, \cdot)  \min \left \{\tilde Z^n_t \tilde D^n _t, M \right \} d\kappa_t\right]\rightarrow  
\mathbb{E}\left [\int _0^T  1_O(t, \cdot)  \min \left \{\frac{\widehat Y_t}{y}, M \right \} d\kappa_t\right].
\end{equation} 
Recall that the stochastic clock has a density $d\kappa_t=\varphi  _t dK_t$ with respect to the deterministic clock $dK _t$.  For each fixed $t$, from \eqref{conv-tau} we have
$$1_O(t, \cdot)  \min \left \{\tilde Z^n_t \tilde D^n _t, M \right \} \varphi _t\rightarrow 1_O(t, \cdot)  \min \left \{\widehat Z_t \widehat D _t, M \right \} \varphi _t,\quad \textrm{in}-\mathbb{P}.$$
Recalling that $\mathbb{E}[\varphi _t]<\infty$ we have, for fixed $t$, that 
$$ M \times  \mathbb{E}[\varphi _t]\geq \mathbb{E}\left [1_O(t, \cdot)  \min \left \{\tilde Z^n_t \tilde D^n _t, M \right \}  \varphi _t \right ]\rightarrow \mathbb{E}\left [1_O(t, \cdot)  \min \left \{\widehat Z_t \widehat D _t, M \right \} \varphi _t\right ].
$$
We integrate the above with respect to the deterministic clock $dK_t$ to obtain
\begin{equation}
\begin{split}
\mathbb{E}\left [\int _0^T  1_O(t, \cdot)  \min \left \{\tilde Z^n_t \tilde D^n _t, M \right \} d\kappa_t\right] =\int _0^T \mathbb{E}\left [1_O(t, \cdot)  \min \left \{\tilde Z^n_t \tilde D^n _t, M \right \}  \varphi _t \right ]dK_t\rightarrow \\
\rightarrow \int _0^T \mathbb{E}\left [1_O(t, \cdot)  \min \left \{\widehat Z_t \widehat D _t, M \right \}  \varphi _t \right ]dK_t=
\mathbb{E}\left [\int _0^T  1_O(t, \cdot)  \min \left \{\widehat Z_t \widehat D _t, M \right \} d\kappa_t\right]\end{split}
\end{equation}
Together with \eqref{hatY} we have
$$\mathbb{E}\left [\int _0^T  1_O(t, \cdot)  \min \left \{\frac{\widehat Y_t}{\hat y}, M \right \} d\kappa_t\right]=\mathbb{E}\left [\int _0^T  1_O(t, \cdot)  \min \left \{\widehat Z_t \widehat D _t, M \right \} d\kappa_t\right],$$
which holds for any optional set $O$ and any bound $M$, therefore $\widehat Y= y\widehat Z\widehat D$, $(d\kappa\times \mathbb P)$-a.e.
  \end{proof}
\begin{proof}[Proof of the Theorem \ref{ExactSlackness}]
Since$$\sum _{k=n}^{N(n)}\alpha _{n,k}(D^k_t)^{-1}\rightarrow (\widehat D_t )^{-1}, \quad t\in[0,T], $$
recalling that $\widehat A$ was defined from $\widehat D$ and \eqref{a-tilde-n}. 
As the processes $A^n$, therefore $\tilde A^n$ only increase by jumps.
Therefore 
we get
 \begin{equation}\label{11291}
 \begin{array}{rcl}
 \Delta \tilde A^n_s &= & \frac{1}{\sum\limits_{k=n}^{N(n)}\alpha _{n,k}\frac{1}{D^k_{s}}} - \frac{1}{\sum\limits_{k=n}^{N(n)}\alpha _{n,k}\frac{1}{D^k_{s+}}}\\
  &= & \sum\limits_{k=n}^{N(n)}\underbrace{\frac{\alpha _{n,k}\frac{1}{D^k_{s+}}}{\sum\limits_{k=n}^{N(n)}\alpha _{n,k}\frac{1}{D^k_{s+}}}}_{\leq 1}\frac{\Delta A^k_s}{D^k_{s}}  
  \underbrace{\frac{1}{\sum\limits_{k=n}^{N(n)}\alpha _{n,k}\frac{1}{D^k_{s}}}}_{\leq 1}.\\
 \end{array}
 \end{equation}
 As $\frac{\alpha _{n,k}\frac{1}{D^k_{s+}}}{\left(\sum\limits_{k=n}^{N(n)}\alpha _{n,k}\frac{1}{D^k_{s+}}\right)}\leq 1$ and $\sum\limits_{k=n}^{N(n)}\alpha _{n,k}\frac{1}{D^k_{s}} \geq 1$, we can bound the latter term in \eqref{11291} by $\sum\limits_{k=n}^{N(n)}\frac{\Delta A^k_s}{D^k_{s}}$ for $s\in[0,T]$. 
 We deduce that
 $$\Delta \tilde A^n_s \leq \sum\limits_{k=n}^{N(n)}\frac{\Delta A^k_s}{D^k_{s}},\quad s\in[0,T].$$ 
Therefore, we have
$$\int_{0}^{t} \widehat V_u d\tilde A^n _u  \leq \sum _{k=n}^{N(n)}(D^k_{t})^{-1} \int_{0}^{t} \widehat V_udA^k_u,\quad t\in[0,T],$$
and thus, we obtain
\begin{equation}\label{main-inez}
\left(\min _{k=n, \dots, N(n)} (Z^k_tD^k_t) \right) \int_{0}^{t} \widehat V_u d\tilde A^n _u  \leq \sum _{k=n}^{N(n)}(Z^k_tD^k_t)(D^k_{t})^{-1} \int_{0}^{t} \widehat V_udA^k_u= \sum _{k=n}^{N(n)}Z^k_t\int_{0}^{t} \widehat V_udA^k_u,\quad t\in[0,T].
\end{equation}
Since the process
$$L^n_t:=\sum _{k=n}^{N(n)}Z^k_t\int_{0}^{t} \widehat V_udA^k_u, \quad 0\leq t\leq T$$
is a non-negative right-continuous  submartingale, the maximal inequality  and  \eqref{fast} together imply
 $$\mathbb{P}\left  (\sup _{0\leq t\leq T}L^n_t\geq \frac 1n \right )\leq n\mathbb{E}[L^n_T]\leq 2^{-n},
$$
so 
\begin{equation}
\label{l}\sup _{0\leq t\leq T}L^n_t\rightarrow 0,\quad \mathbb{P}-a.s.
\end{equation}
Since $Z^nD^n \rightarrow \frac{1}{y}\widehat Y>0$, $(d\kappa\times \mathbb P)$-a.e.,  consequenlty 
$$\min _{k=n, \dots, N(n)} (Z^kD^k)\rightarrow \frac{1}{y}\widehat Y>0\quad \left(d\kappa \times \mathbb{P}\right)-a.e.,$$
we obtain from\eqref{main-inez} and \eqref{l} that 
the increasing RC process
$$\tilde L^n_t:=\int _0^t \widehat V_u d\tilde A^n_u,\quad 0\leq t\leq T,$$
converges to zero in the product space.  Denoting by $\mathcal{O}\subset \Omega\times [0,T]$ the exceptional set where convergence to zero does not take place,  and taking into account that $\tilde L^n$ are increasing,  we have that for 
for $$T^{\mathcal{O}}(\omega)=\inf \{0\leq t\leq T: (\omega ,t)\in \mathcal{O}\},$$
we have 
$$(T^{\mathcal{O}},T]\subset \mathcal{O}.$$
Now 
$$(d\kappa\times \mathbb P)(\mathcal{O})=0, $$
implies that 
$T^{\mathcal{O}}\geq T,~\mathbb{P}-a.s.$ (here used an assumption that $T$ is the minimal time horizon in the sense that the deterministic clock $K$ is such that $K_t<K_T$, for every $t\in[0,T)$, this assumption does not restrict generality). Thus, there exists a 
 a set $\Omega ^*$  of full probability $\mathbb{P}(\Omega ^*)=1$  such that, for each $\omega \in \Omega ^*$ and $t<T$ we have
 $$\int _0^t \widehat V_u (\omega)d\tilde A^n_t (\omega)\rightarrow 0.$$

Let us fix an $\omega\in \Omega ^*$ and such that, 
for this $\omega$, the probability measure 
$d\tilde A^{n}(\omega)$ converges weakly to $d\widehat A(\omega)$ over the interval $[\theta _0(\omega), T]$.
The set of such $\omega$'s still has probability $1$. The Skorokhod representation theorem asserts that there exists a new probability space  $\Omega ^{\omega}$ and a  sequence of random times 
$(t^{n}(\omega))_{n\in\mathbb N}$ as well as $\widehat t (\omega)$ such that the distribution of $t^{n} (\omega)$ is $d\tilde A^{n}$,  the distribution of $\widehat t (\omega)$ is $d \widehat A (\omega)$ and
$$t^{n}(\omega)\rightarrow \widehat t (\omega),\quad \mathbb{P}^{\omega}-a.s. $$
on the new, artificial, probability space. Fix $t<T$. We have
$$\mathbb{E}^{\omega}[\widehat V _{t^{n}(\omega)} (\omega) 1_{\{t^n(\omega)\leq t\}}]=\int _{\theta _0}^t \widehat V_u (\omega)  d\tilde A^{n} _u(\omega) \rightarrow 0.$$
One can see that 
$$\xi (\omega):=\liminf _n \widehat V _{t^{n}(\omega)}(\omega)\in \{ \widehat V _{\widehat t(\omega)-}(\omega),  \widehat V _{\widehat t(\omega)}(\omega)\}, \quad \mathbb{P}^{\omega}-a.s.$$ 
and 
$$1_{\{t^n(\omega)\leq t\}}\rightarrow 1_{\{\widehat t (\omega)\leq t\}}\  \textrm{on} \ \ \{\widehat t (\omega)<t\},\ \mathbb{P}^{\omega}-a.s.$$ 
Therefore,  
applying Fatou's lemma on $\Omega ^{\omega}$, we obtain
$$\mathbb{E}^{\omega}[\xi (\omega) 1_{\{\widehat t(\omega)< t\}}]=0.$$
We recall that both $\widehat V (\omega)$ and $\widehat V_{-}(\omega)$ are nonnegative, consequently we have
$$0\leq \min \{ \widehat V _{\widehat t(\omega)-}(\omega),  \widehat V _{\widehat t(\omega)}(\omega) \}\leq \xi (\omega),$$ and therefore we get 
$$\mathbb{E}^{\omega}[ \min \{ \widehat V _{\widehat t(\omega)-}(\omega),  \widehat V _{\widehat t(\omega)}(\omega) \}1_{\{\widehat t(\omega)< t\}}]=0.$$
This means that the distribution  of $\widehat t (\omega)$ over the interval $[\theta _0(\omega), t)$ only charges the complement of the set of times 
$$\left\{u<t: \quad \widehat V_{u-}(\omega)=0\quad or\quad \widehat V_u (\omega)=0\right\}.$$
By taking $t\rightarrow T$, one completes the proof.
\end{proof}
\bibliographystyle{alpha}
\bibliography{finance2}
\end{document}